\documentclass[a4paper]{amsart}

\usepackage[labelfont=bf, textfont=normal,
			justification=justified,
			singlelinecheck=false]{caption} 
\usepackage{graphicx} 
\usepackage{hyperref} 
\usepackage{bbm} 
\usepackage[top=2.9cm, bottom=2.9cm,
			inner=2.5cm, outer=2.5cm]{geometry}			
\usepackage{parskip} 

\usepackage{enumerate}

\usepackage{pgfplots}
\pgfplotsset{compat=1.16}

\usepackage{subcaption}
\usepackage{pgf,tikz}
\usetikzlibrary{spy, external}
\usepackage{caption}
\usepackage{todonotes}

\usepackage{float}

\newcommand{\triple}{{\vert\kern-0.25ex\vert\kern-0.25ex\vert}}

\usepackage{comment}
\newcommand*\dd{\mathop{}\!\mathrm{d}}
\newcommand{\R}{\mathbb{R}}
\newcommand*{\E}{\mathbb{E}}
\newcommand{\M}{\mathcal{M}}

\newtheorem{theorem}{Theorem}[section]
\newtheorem{proposition}{Proposition}[section]
\newtheorem{lemma}{Lemma}[section]

\newtheorem{remark}{Remark}[section]

\theoremstyle{definition}

\newtheorem{corollary}{Corollary}[section]

\newcommand{\defeq}{\mathrel{:\mkern-0.25mu=}}

\DeclareMathOperator\supp{supp}
\DeclareMathOperator\spann{span}

\newcommand{\Y}{\mathbb{Y}}

\definecolor{crimson2143940}{RGB}{214,39,40}
\definecolor{darkgray176}{RGB}{176,176,176}
\definecolor{darkorange25512714}{RGB}{255,127,14}
\definecolor{forestgreen4416044}{RGB}{44,160,44}
\definecolor{lightgray204}{RGB}{204,204,204}
\definecolor{mediumpurple148103189}{RGB}{148,103,189}
\definecolor{sienna1408675}{RGB}{140,86,75}
\definecolor{steelblue31119180}{RGB}{31,119,180}

\usepackage{amssymb}
\usepackage{xcolor}

\usepackage{color}

\definecolor{OIBlue}{RGB}{0,114,178}
\definecolor{OISkyBlue}{RGB}{86,180,233}
\definecolor{OIGreen}{RGB}{0,158,115}
\definecolor{OIOrange}{RGB}{230,159,0}
\definecolor{OIVermilion}{RGB}{213,94,0}
\definecolor{OIPurple}{RGB}{204,121,167}

\pgfplotsset{
  compat=1.18,
  every axis plot/.append style={line join=round},
}

\tikzset{with marks/.style={mark repeat=4, mark phase=2, mark size=1.5pt}}

\tikzset{with marks long/.style={mark repeat=40, mark phase=20, mark size=1.5pt}}

\tikzset{with marks every/.style={mark repeat=1, mark phase=1, mark size=1.5pt}}

\tikzset{method LogBSDEF/.style={color=OIBlue, very thick, mark=*, mark options={solid}}}
\tikzset{method BSDEF/.style={color=OISkyBlue, very thick, mark=o, mark options={fill=white,solid}}}

\tikzset{method LogDSF/.style={color=OIPurple, very thick, mark=square*, mark options={solid}}}
\tikzset{method DSF/.style={color={OIPurple!70!black}, very thick, mark=square, mark options={fill=white,solid}}}

\tikzset{method PF 1e6/.style={color={OIVermilion!95!black}, thick, solid,  mark=triangle*, mark options={solid}}}
\tikzset{method PF 1e5/.style={color={OIVermilion!80!black}, thick, dashed, mark=triangle*, mark options={solid}}}
\tikzset{method PF 1e4/.style={color={OIVermilion!60!black}, thick, dotted, mark=triangle*, mark options={solid}}}

\tikzset{method EnKF 1e6/.style={color={OIGreen!90!black}, thick, solid,  mark=diamond*, mark options={solid}}}
\tikzset{method EnKF 1e5/.style={color={OIGreen!70!black}, thick, dashed, mark=diamond*, mark options={solid}}}
\tikzset{method EnKF 1e4/.style={color={OIGreen!55!black}, thick, dotted, mark=diamond*, mark options={solid}}}

\tikzset{method EKF/.style={color={OIOrange!90!black}, thick, dashdotted, mark=star, mark options={solid}}}
\tikzset{method Reference/.style={color=black!60, ultra thick, mark=none}}

\tikzset{method KF density/.style={color={OIGreen!55!black}, thick, dashed}}

\tikzset{method PF 1e4 density/.style={color={OIVermilion!60!black}, thick, dotted}}

\tikzset{method EnKF 1e4/.style={color={OIGreen!55!black}, thick, dotted, mark=diamond*, mark options={solid}}}

\tikzset{method LogDSF density/.style={color=OIPurple, thick}}

\tikzset{method LogBSDEF density/.style={color=OIBlue, very thick, dotted}}
\tikzset{method LogBSDEF-LSTM density/.style={color=OISkyBlue, thick, dashdotted}}

\usepackage{mathtools}  
\mathtoolsset{showonlyrefs}

\title[A convergent scheme for the Fokker--Planck equation with deep splitting]{
A convergent scheme for the Bayesian filtering problem based on the Fokker--Planck equation and deep splitting}

\begin{document}
\author[K.~B{\aa}gmark]{Kasper B{\aa}gmark}
\address{Kasper B{\aa}gmark\\
Department of Mathematical Sciences\\
Chalmers University of Technology and University of Gothenburg\\
SE--412 96 Gothenburg\\
Sweden}
\email{bagmark@chalmers.se}

\author[A.~Andersson]{Adam Andersson} 
\address{Adam Andersson\\ 
Department of Mathematical Sciences\\
Chalmers University of Technology and University of Gothenburg\\
S--412 96 Gothenburg, Sweden\\
and Saab AB\\
S--412 76 Gothenburg, Sweden} 
\email{adam.andersson@chalmers.se} 

\author[S.~Larsson]{Stig Larsson}
\address{Stig Larsson\\
Department of Mathematical Sciences\\
Chalmers University of Technology and University of Gothenburg\\
SE--412 96 Gothenburg\\
Sweden}
\email{stig@chalmers.se}

\author[F.~Rydin]{Filip Rydin}
\address{Filip Rydin\\
Department of Electrical Engineering\\
Chalmers University of Technology and University of Gothenburg\\
SE--412 96 Gothenburg\\
Sweden}
\email{filipry@chalmers.se}


\begin{abstract}
    A numerical scheme for approximating the nonlinear filtering density is introduced and its convergence rate is established, theoretically under a parabolic H\"{o}rmander condition, and empirically in numerical examples. In a prediction step, between the noisy and partial measurements at discrete times, the scheme approximates the Fokker--Planck equation with a deep splitting scheme, followed by an exact update through Bayes' formula. This results in a classical prediction-update filtering algorithm that operates online for new observation sequences post-training. The algorithm employs a sampling-based Feynman--Kac approach, designed to mitigate the curse of dimensionality. As a corollary we obtain the convergence rate for the approximation of the Fokker--Planck equation alone, disconnected from the filtering problem. The convergence analysis is complemented by a nonlinear $10$-dimensional numerical example demonstrating the robustness of the method.
\end{abstract}

\maketitle

\section{Introduction}
Approximate nonlinear filters often rely on assumptions of unimodality (e.g., Kalman filters \cite{Kalman_Bucy}) or computationally cheap tractable simulations. The latter mainly refers to sequential Monte Carlo methods, also known as particle filters. Despite their effectiveness in various settings and their asymptotic convergence to the true filter, these methods suffer from the curse of dimensionality with respect to the state dimension \cite{Quinn, Rebeschini_2015, snyder2011particle, snyder2015performance}. There is currently a lot of progress within the Bayesian filtering community, e.g., \cite{luk2024learning}, where sampling-based methods have improved substantially \cite{ corenflos2024particlemala, corenflos2024conditioning, finke2023conditional, 
naesseth2019high, schauer2017guided, zhao2024conditional}. However, the curse of dimensionality remains an open problem for the nonlinear case. 

One motivation for developing methods suitable for a high-dimensional setting comes from applications. There exist numerous domains of applications for Bayesian filtering:
target tracking \cite{blackman1999design, cui2005comparison, goodman1997mathematics}, 
finance \cite{Date,zeng2013state},
chemical engineering \cite{Rutzler}, and weather forecasts \cite{Cassola,Duc_Kuroda, galanis2006applications}, to mention a few prominent ones. Many problems are inherently high-dimensional. An extreme example is \cite{apte2008data}, in which the authors discuss the challenge in weather forecasting with a corresponding dimension of $10^7$ that arises from the spatial domain.

To address the challenge of high-dimensional applications, effective methods for Bayesian filtering must be developed to manage the complexity of the state space and accurately estimate the filtering probability distribution. In many settings, it is desirable to obtain the probability density of the filter and not only estimates of mean and covariance, or a finite empirical distribution. We refer to the filtering probability distribution as the probability of observing a hidden state $S_k$ given noisy measurements $y_{0:k}$ up to time $k$. If such a density $p_k$ exists, it satisfies, for a measurable set $C$ in $ \R^d$, the relation
\begin{align} 
\label{introduction: filtering density}
    \mathbb{P}
    (S_k\in C 
    \mid
    y_{0:k}
    )
    =
    \int_C 
    p_k
    (x 
    \mid 
    y_{0:k}
    ) 
    \dd x.
\end{align}
Moreover, it can be shown that this density adheres to an evolution equation: for discrete observations, it satisfies a Fokker--Planck equation \cite{Kloeden_Platen}, and for continuous observations, it satisfies the Kushner--Stratonovich equation \cite{Kushner}. There are approximation techniques that work well in low dimensions to obtain the filtering density through solving these (Stochastic) Partial Differential Equations (PDEs). These classical techniques, such as finite elements \cite{brenner2007mathematical} and finite differences, are real-time efficient only in one dimension and stop to be computationally feasible in state dimension $d\geq 4$, which is insufficient for most applications.

In recent years with the explosion of machine learning in applied mathematics, there has been a lot of work going towards efficient deep learning solutions to PDEs. Direct approaches such as Physics-Informed Neural Networks (PINN) \cite{raissi2019physics}, deep Backward Stochastic Differential Equation (BSDE) methods \cite{andersson2023convergence,E_2017}, deep neural operators \cite{Lu_2021}, deep Ritz method \cite{e2017deep}, deep Picard iterations \cite{han2024deep}, and more, have shown excellent performance in handling higher dimensions to different degrees. Recently, there was success with a score-based PINN applied to a high-dimensional Fokker--Planck equation in \cite{hu2024score}. However, in general, there are challenges with these methods with respect to modeling probability densities. Specifically, because the Fokker--Planck equation models a probability density function associated with Brownian motion and, since in very high dimensions, its values become extremely small, this poses challenges for existing methods in terms of numerical accuracy.

In this paper, we focus on further development of one such approximation technique, called Deep Splitting. It was demonstrated to perform very well for PDEs with strong symmetries in dimensions up to $10\,000$ \cite{Arnulf_PDE}, and has been extended to partial integro-differential equations \cite{frey2022convergence}. For the Zakai equation \cite{Zakai}, different deep splitting methods have been applied for offline filtering, when the solution is approximated for a single observation sequence \cite{Arnulf, Crisan_Lobbe, lobbe2022deep}. In a subsequent work, deep splitting was combined with an energy-based approach on the Zakai equation and generalized to the online setting where it showed promising results for a state dimension up to $d=20$ \cite{bagmark_1}. Here we extend it to a format more interesting for most applications in engineering and science, namely the setting where the underlying process $S$ is continuous
in time and space, but the observations $O$ are discrete in time. This is done by propagating the Fokker--Planck equation for the prediction and by updating the solution with Bayes' formula at every measurement. The code for the implementation is publicly available\footnote{\url{https://github.com/bagmark/deep-density-filtering}}, and in a concurrent work \cite{baagmark2025high} we benchmark the method against other deep density methods and classical methods in a high-dimensional setting.

\subsection{Problem formulation}
Let $B$ be a Brownian motion on some complete probability space, $T>0$ denote a terminal time for the filtering, and $0=t_0<t_1<\dots<t_K=T$ be observation times.
The filtering problem in question deals with time-continuous state and time-discrete observations, where the state and observation models are given by
\begin{align}
\label{introduction: system of equations}
\begin{split}
    \dd
    S_t 
    &=  
    \mu(S_t)
    \,\mathrm{d} t 
    + 
    \sigma(S_t) \,\mathrm{d} B_t,
    \quad
    t\in(0,T],
    \\
    S_0 
    &\sim 
    q_0,
    \\
    O_k
    &\sim 
    g(O_k \mid S_{t_k}) 
    ,\quad
    k=0,\dots,K.  
\end{split}
\end{align}
In this setting we refer to the $\R^d$-valued process $S$, solving a Stochastic Differential Equation (SDE) with drift $\mu$ and diffusion coefficient $\sigma$, as the unknown state process and the $\R^{d'}$-valued variable $O=(O_k)_{k=0}^K$ as the observation process. We stress that \eqref{introduction: system of equations} is a statistical model, only used in a distributional sense, meaning that pathwise values of $S$ and $O$ are irrelevant and therefore notation for the probability space is not introduced. In any sensible statistical procedure, the distribution of $O$, determined by the model \eqref{introduction: system of equations} is approximately that of the data $y$, as otherwise the filter will output nonsense. The model implicitly defines a measurement likelihood $L(y_k,x) \defeq p(O_k = y_k \mid S_{t_k}=x)=g(y_k\mid x)$, and in Section \ref{sec:setting} we establish rigorous regularity conditions for the measurement likelihood $L(y,x)$. The standard measurement model, used in almost all publications on filtering with discrete observations, is the Gaussian model with $g(y_k\mid x)=N(y_k;h(x),Q)$, where $h$ is a measurement function and $Q$ is a covariance matrix. The filtering problem, i.e., the problem of obtaining the filtering density \eqref{introduction: filtering density}, assuming the state and observation model \eqref{introduction: system of equations}, and given data $y$, can be solved by recursively solving a Fokker--Planck equation and by updating using Bayes' formula. To formalize the unnormalized exact filter we initialize $p_0(0,x,y_0)=q_0(x)L(y_0,x)$, and recursively define
\begin{align}
\begin{split}
\label{introduction: Fokker--Planck}
    p_{k}(t,x,y_{0:k}) 
    &= 
    p_{k}(t_{k},x,y_{0:k})
    +
    \int_{t_{k}}^{t} 
    A^* p_{k}(s,x,y_{0:k}) 
    \dd s
    ,\quad
    t\in (t_{k},t_{k+1}], 
    \ k=0,\dots,K-1,
    \\
    p_{k}
    (t_k,x,y_{0:k})
    &=
    p_{k-1}
    (t_k,x,y_{0:k-1})
    L(y_k,x)
    ,\quad
    k=1,\dots,K.
\end{split}
\end{align}
Throughout this paper the filter, $p_k$, is denoted with three arguments $(t,x,y)$, rather than the conventional $(t,x\mid y)$. This should not be confused with the joint distribution over $(t,x,y)$, which we do not address here. The distribution is unnormalized, since the normalizing factor in the Bayes formula has been neglected, but is otherwise a classical Bayesian filtering system of recursive prediction-update steps.
The operator $A^*$, that appears in the Fokker--Planck equation in \eqref{introduction: Fokker--Planck}, is the adjoint of the infinitesimal generator of the diffusion process $S$ in \eqref{introduction: system of equations}. It remains to solve \eqref{introduction: Fokker--Planck} to define $(p_k)_{k=0}^{K-1}$ including a final update at time $t_K$ to define $p_K$. In Section~\ref{Section: Method} we introduce a Feynman--Kac based solution $\widetilde{\pi}$, that approximates the exact filtering density $p$. The approximation scheme, where $N\in\mathbb{N}$ denotes the number of discretization steps between consecutive observation times, is recursively defined by
\begin{align*}
    \begin{split}
        (\widetilde{\pi}_{k,n+1}
        (x,y))_{(x,y)\in \mathbb{R}^d\times \mathrm{supp}(\mathbb{P}_Y)}  
        &=
        \mathop{\mathrm{arg\,min}}_{u\in 
        L^\infty(\mathrm{supp}(\mathbb{P}_Y);
        C(\mathbb{R}^d;\mathbb{R}))}
        \mathbb{E}
        \bigg[ \Big| 
        u(Z_{N-(n+1)},Y_{0:{k}})
        -
        G_b
        \widetilde{\pi}_{k,n}
        (Z_{N-n},Y_{0:k})
        \Big|^2 \bigg],
        \\
        &\hspace{4em} 
        k=0,\dots,K-1,
        \
        n=0,\dots,N-1,
        \\
        \widetilde{\pi}_{0,0}(x,y_{0}) 
        &= 
        q_0(x)L(y_0,x),
        \\
        \widetilde{\pi}_{k,0}(x,y_{0:k}) 
        &=
        \widetilde{\pi}_{k-1,N}(x,y_{0:k-1})
        L(y_k,x),
        \quad 
        k=1,\dots,K.
    \end{split}    
\end{align*}
Here $Z$ is the Euler--Maruyama approximation of an auxiliary process $X$. The process $X$ satisfies an SDE with drift $b$ and diffusion $\sigma$, driven by a Brownian motion $W$ on $(\Omega,\mathcal A,\mathbb{P})$. The optimization problem is posed over $L^\infty(\supp(\mathbb P_Y);C(\R^d;\R))$, i.e., the space of (essentially) bounded functions on the support of the probability distribution $\mathbb P_Y$ with values in $C(\R^d;\R)$. 
Here $C(\R^d;\R)$ denotes the space of continuous functions $\R^d\to\R$. The operator $G_b$ is a first order differential operator. By solving the optimization problem and finding $\widetilde{\pi}_{k}$ for all $k=0,\dots,K$, one has obtained a machine for approximating the unnormalized filtering density (and prediction density) for all possible measurement sequences $y\in \supp(\mathbb{P}_Y)\subset \Y \defeq \R^{d'\times (K+1)}$ with probability distribution $\mathbb{P}_Y$. This makes inference in the online setting very efficient.

The main contribution of this paper is the introduction of this new scheme and its corresponding error analysis. We derive a strong convergence rate in $L^\infty(\Y;L^\infty(\R^d;\R))$ of order $1$ in time, assuming a parabolic Hörmander condition as in \cite{cattiaux2002hypoelliptic}. The convergence rate is independent of the distribution of the data $y$. Furthermore, we demonstrate the method on a nonlinear $10$-dimensional example to complement the convergence study.

This paper is organized as follows: In Section~\ref{Section: Method} we present the notation and setting that we use and derive the approximations. In this section we also show regularity and uniqueness of the solution to \eqref{introduction: Fokker--Planck}, and of its approximations, and state the main theorem on strong convergence. Section~\ref{section: error analysis} contains the proof of the strong convergence and some auxiliary lemmas that we need.  Finally, Section~\ref{section: numerical experiments} details the additional approximation steps needed to obtain a tractable method based on neural networks and Monte Carlo simulation, and presents the numerical experiments.

\section{The online deep splitting filter and main results}
\label{Section: Method}

\subsection{Notation}
We denote by $\langle x,y \rangle$ and $\|z\|$ the inner product and norm in the Euclidean space $\mathbb{R}^d$ if $x,y,z\in \R^d$ and the Frobenius norm if $z\in \R^{d\times d}$. 
The space of functions in $[0,T]\times\mathbb{R}^{d}\to\mathbb{R}$, which are $k$ times continuously differentiable in the first variable and $n$ times continuously differentiable in the second variable with no cross derivatives between the variables, is denoted $C^{k,n}([0,T]\times \mathbb{R}^d; \mathbb{R})$. Furthermore, functions in the space $C_{\mathrm{b}}^{k,n}([0,T]\times \mathbb{R}^d; \mathbb{R})$ have bounded derivatives. The space $C_{\mathrm{p}}^k(\R^d;\R)$ consists of functions in $C^k(\R^d;\R)$ which, together with their derivatives, are of most polynomial growth. 
Similarly, functions in the space $C_{0}^k(\R^d;\R)\subset C^k(\R^d;\R)$ are such that they and their derivatives tend to $0$ at infinity. For $n$ or $k=\infty$, we adopt the usual modifications.
Let $(A,\mathcal{B},\nu)$ be a measure space and $U$ be a Banach space. By $\mathcal{L}^0(A;U)$, we denote the space of strongly measurable functions $f\colon A\to U$ and by $L^0(A;U)$ the equivalence classes of functions in $\mathcal{L}^0(A;U)$ that are equal $\nu$-almost everywhere. The Bochner spaces $L^p(A;U)\subset L^0(A;U)$, $p\in[1,\infty]$, are defined by
\begin{align*}
    \|
    f
    \|_{L^p(A;U)}
    &\defeq
    \Big(
    \int_A
    \|
    f(x)
    \|_U^p
    \dd 
    \nu(x)
    \Big)^{\frac{1}{p}}
    <
    \infty,
    \quad
    p\in[1,\infty),
    \\
    \|
    f
    \|_{L^\infty(A;U)}
    &\defeq
    \sup_{x\in A}
    \|
    f(x)
    \|_U
    <
    \infty.
\end{align*}
Here and throughout the paper, we write $\sup$ to mean the essential supremum. For Banach spaces $U$ and $V$, we denote by $\mathcal{L}(U;V)$ the space of bounded linear operators that map from $U$ to $V$. 

For smooth vector fields $V,W\in C^1(\R^d;\R^d)$ we define the Lie bracket of $V$ and $W$ by
\begin{align*}
    \big[V,W\big](x)
    =
    \mathrm{D}W(x)V(x)
    -
    \mathrm{D}V(x)W(x)
    ,\quad
    x\in\R^d
    .
\end{align*}
Here $\mathrm{D}V(x)$ is the Jacobi matrix of $V$ with respect to $x\in\R^d$. We say that vector fields $V_0,\dots,V_n \in C_{\mathrm{b}}^\infty(\R^d;\R^d)$ satisfy a parabolic Hörmander condition if, for all $x\in\R^d$, we have 
\begin{align}
\begin{split} \label{eq: parabolic hörmander}
    &\spann \Big\{
    V_{j_0}(x),
    \big[V_{j_1}(x)
    ,V_{j_2}(x)\big],
    \Big[\big[V_{j_1}(x)
    ,V_{j_2}(x)\big]
    ,V_{j_3}(x)\Big],
    \dots ;
    \\
    &\hspace{12em}
    j_0\in \{1,\dots,n\},\,
    j_i \in \{0,\dots,n\}\text{ for all } 
    i=1,\dots,n
    \Big\}
    =
    \R^d.
\end{split}
\end{align}
To improve clarity, we suppress $x\in\R^d$ in equations throughout the paper whenever appropriate.
We define the index set $\mathcal{I}_{K,N} \defeq \{0,\dots,K-1\}\times \{0,\dots,N\} \cup \{K\}\times\{0\}$, for $K,N\in \mathbb{Z}_+$.
Finally, if $R$ is a stochastic process, then we let $R^{t,x}$ denote the conditioned process $R$ that starts in $x$ at time $t$, so that $R_t^{t,x}=x$.

\subsection{Setting}
\label{sec:setting}
Throughout the rest of Sections~\ref{Section: Method} and \ref{section: error analysis} we assume that the following holds.
Let $T>0$, and $d,d',K\geq1$ be integers. We consider $K+1$ uniformly distributed observation times, denoted $(t_k)_{k=0}^K$, satisfying 
\begin{align}
    0=t_0<t_1<\dots<t_{K-1}<t_K=T,
\end{align}
with $t_{k+1}-t_k=\frac{T}{K}$ for  $k=0,\dots,K-1$. Moreover, for numerical approximations we use a family of finer uniform grids given by
\begin{align}
    t_{k}
    =
    t_{k,0} 
    < 
    t_{k,1} 
    < 
    \dots 
    <
    t_{k,N-1}
    <
    t_{k,N}
    =
    t_{k+1},
    \quad
    k=0,\dots,K-1,
\end{align}
with $t_{k,n+1}-t_{k,n}=\frac{T}{KN}$ for  $k=0,\dots,K-1$, $n=0,\dots,N-1$. 
In our error analysis, we write $\tau \defeq \frac{T}{KN}$, where thus $K$ is fixed while $N$ tends to infinity.

The underlying state process $(S_t)_{t\in[0,T]}$ is $\R^d$-valued and, for convenience, we define $\mathbb{Y} \defeq \R^{d'\times (K+1)}$ to denote the space of data measurements. We denote by $y_{k:n}$, $0\leq k\leq n$, the $d'\times (n-k+1)$-matrix $(y_{k},y_{k+1},\dots,y_{n})$.

Throughout the paper, we use a filtered probability space $(\Omega,\mathcal A, (\mathcal{F}_t)_{t\in[0,T]}, \mathbb P)$ equipped with a Brownian motion $W$ adapted to $(\mathcal{F}_t)_{t\in[0,T]}$.

The assumptions on the functions in the statistical model \eqref{introduction: system of equations} are listed next.
\begin{enumerate}[\hspace{-0.5em}(i)\hspace{0.5em}]
    \item \label{assumption: coefficient bound and lipschitz}
    The coefficients $\mu$ and $\sigma$, initial density $q_0$, and measurement likelihood $L$ are bounded, infinitely smooth and with bounded derivatives, i.e., $\mu\in C_{\mathrm{b}}^{\infty}(\R^d;\R^d)$, $\sigma\in C_{\mathrm{b}}^\infty(\R^d;\R^{d\times d})$, $q_0\in C_{\mathrm{b}}^\infty(\R^d;\R)$, $L\in C_{\mathrm{b}}^\infty(\R^{d'}\times \R^d;\R)$.   

    \item \label{assumption: Hörmander}
    The coefficients $\mu$ and $\sigma=[\sigma_1, \dots, \sigma_d]$, where $\sigma_i$ denotes the $i$'th column of $\sigma$, satisfy the parabolic Hörmander condition, i.e., the vector fields $V_0,\dots,V_d$ defined by
    \begin{align*}
            V_i
            =
            \sigma_i
            ,\quad
            \text{for }
            i=1,\dots,d,
            \qquad
            V_0
            =
            \mu
            +
            \frac{1}{2}
            \sum_{j=1}^d
            \mathrm{D}V_j
            V_j,
    \end{align*}
    satisfy \eqref{eq: parabolic hörmander} for all $x\in\R^d$.
\end{enumerate}

We remark that it is usual to assume that the coefficients $\mu$ and $\sigma$ have bounded derivatives and at most linear growth.  This is sufficient for most of this paper also, but the proof of our convergence result seems to require that the coefficients are bounded.

\subsection{Solution to the filtering problem and deep splitting approximations}
The filtering method considered in this paper is a version of the method in \cite{bagmark_1, Arnulf} but now adapted to the case of discrete observations. More precisely, we model the state, denoted $S$, with an SDE and the observation process, denoted $O$, with discrete random variables coupled to discrete time points of $S$, see \eqref{introduction: system of equations}. Under the conditions of Section~\ref{sec:setting}, it is well known from the literature that the SDE in \eqref{introduction: system of equations} has a unique solution $S$. That the observation process $O$ is well defined is clear. We remind the reader that the problem under consideration is that of finding the conditional probability distribution of $S_{t}$ given the measurements $y_{0:k}$, up to and including time $t$. More precisely, we are for $y\in\mathbb{Y}$ interested in the unnormalized conditional density $p(t,x\mid y_{0:k})$, $t\in[t_k,t_{k+1})$, satisfying for all measurable sets $C$ in $\R^d$ the relation
\begin{align*}
    \mathbb{P}
    (S_t\in C 
    \mid
    y_{0:k}
    )
    =
    \frac{
    \int_C 
    p(t,
    x 
    \mid 
    y_{0:k}
    ) 
    \dd x
    }{
    \int_{\R^d} 
    p(t,
    x 
    \mid 
    y_{0:k}
    ) 
    \dd x
    }.
\end{align*}
To this end, we introduce the Fokker--Planck equation. We recall that the state process $S$ satisfying \eqref{introduction: system of equations} has an associated infinitesimal generator $A$. This operator and its formal adjoint $A^*$, are defined, with $a\defeq \sigma \sigma^{\top}$, for $\varphi\in C^{\infty}_0(\mathbb{R}^d;\mathbb{R})$, as
\begin{align*}
    A\varphi 
    = 
    \frac{1}{2}\sum_{i,j=1}^d a_{ij}\,
    \frac{\partial^2 \varphi}{\partial x_i \partial x_j} 
    + \sum_{i=1}^d \mu_i \, 
    \frac{\partial \varphi}{\partial x_i}
    \qquad 
    \text{and} 
    \qquad
    A^* \varphi 
    = 
    \frac{1}{2}\sum_{i,j=1}^d 
    \frac{\partial^2}{\partial x_i \partial x_j} 
    (a_{ij}\varphi) 
    - \sum_{i=1}^d 
    \frac{\partial}{\partial x_i} 
    (\mu_i \varphi).
\end{align*}
Similarly, by letting $b$ define a new drift function, we obtain an alternative generator $A_b$ given by
\begin{align}
    A_b
    \varphi 
    &= 
    \frac{1}{2}\sum_{i,j=1}^d a_{ij}\,
    \frac{\partial^2 \varphi}{\partial x_i \partial x_j} 
    +
    \sum_{i=1}^d 
    b_i \, 
    \frac{\partial \varphi}{\partial x_i}
    .
\end{align}
To connect with the framework of the deep splitting method \cite{bagmark_1, Arnulf,Arnulf_PDE}, we write $A^*=A_b+F_b$, where the first order differential operator $F_b$ is, for $\varphi \in C^{1}(\R^d;\R)$, defined by
\begin{align}
\begin{split}
    F_b\varphi 
    =&\, 
    \sum_{i=1}^d 
    \bigg(
    \sum_{j=1}^d
    \frac
    {\partial a_{ij}}
    {\partial x_j }
    \bigg)
    \frac
    {\partial \varphi}
    {\partial x_i }
    + 
    \frac{1}{2}
    \sum_{i,j=1}^d 
    \frac
    {\partial^2 a_{ij}}
    {\partial x_i \partial x_j}
    \,  
    \varphi 
    -
    \sum_{i=1}^d 
    \frac
    {\partial \mu_i}
    {\partial x_i}
    \, 
    \varphi
    \,
    -
    \sum_{i=1}^d 
    \mu_i \, 
    \frac{\partial \varphi}{\partial x_i}
    -
    \sum_{i=1}^d 
    b_i \, 
    \frac{\partial \varphi }{\partial x_i}
    .
\end{split}
\end{align}
We remark that in the original derivation of the deep splitting method, the drift function $b$ was defined as $b = \mu$, implying $A_b = A$.
In our extended framework we allow for different choices satisfying $b\in C_{\mathrm{b}}^\infty$, and in Section~\ref{section: numerical discussion} we discuss numerical impacts of this choice. One choice is
\begin{align}
\label{eq: mu-bar choice}
    b_i
    &=
    -
    \mu_i
    +
    \sum_{j=1}^d
    \frac{\partial a_{ij}}
    {\partial x_j }
    ,
    \quad
    i=1,\dots,d,
\end{align}
for which the first order terms in $F_b$ disappear. This choice will facilitate our error analysis in Section~\ref{section: error analysis}.
This is one reason why we assume that the coefficients are bounded; in fact, $b$ as in \eqref{eq: mu-bar choice} does not have bounded derivatives unless $\sigma$ is bounded together with its derivatives.

The recursion \eqref{introduction: Fokker--Planck}, that defines $(p_k)_{k=0}^K$, then becomes 
\begin{align}
\begin{split}
\label{eq: global Fokker--Planck with update}
    p_{k}(t,x,y_{0:k}) 
    &= 
    p_{k}(t_{k},x,y_{0:k})
    +
    \int_{t_{k}}^{t} 
    A_b p_{k}(s,x,y_{0:k}) 
    \dd s
    +
    \int_{t_{k}}^{t}
    F_b
    p_{k}(s,x,y_{0:k})
    \dd s
    ,\quad
    t\in (t_{k},t_{k+1}], 
    \\
    & \qquad
    k=0,\dots,K-1,
    \\
    p_{0}
    (0,x,y_{0})
    &=
    q_0(x)L(y_0,x),\\
    p_{k}
    (t_k,x,y_{0:k})
    &=
    p_{k-1}
    (t_k,x,y_{0:k-1})
    L(y_k,x)
    ,\quad
    k=1,\dots,K.
\end{split}
\end{align}
The solution $(p_{k})_{k=0}^K$, obtained by solving \eqref{eq: global Fokker--Planck with update} for every path $y_{0:K}\in \Y$, involves both prediction, when $t\in (t_k,t_{k+1}]$, and filtering, when $t=t_k$. We remark that, since $A_b + F_b = A^*$, the solution of \eqref{eq: global Fokker--Planck with update} does not depend on the choice of $b$. The following proposition summarizes the regularity of the solution.
\begin{proposition}
    \label{prop:properties of p}
    There exists a unique $p_{k}\in L^\infty(\Y;C([t_k,t_{k+1}]\times\R^d,\R))$, $k=0,\dots,K$, satisfying \eqref{eq: global Fokker--Planck with update}. Moreover, $p_{k}(y)\in C_{\mathrm b}^{1,\infty}([t_k,t_{k+1}]\times \R^d; \R)$ for all $k=0,\dots,K$, $y\in\Y$.
\end{proposition}
\begin{proof}
    We fix $y\in\Y$ and begin by noting that if $p_{k-1}(t_{k},y_{0:k-1})\in C_{\mathrm{b}}^\infty
    (\R^d;\R)$, 
    then $p_{k}(t_k,y_{0:k})= p_{k-1}(t_{k},y_{0:k-1})L(y_{k})\in C_{\mathrm{b}}^\infty
    (\R^d;\R)$, since $C_{\mathrm{b}}^\infty(\R^d;\R)$ is an algebra. In particular, this holds for $p_{0}(0,y_0)=q_0L(y_0)$. It remains to prove that the Fokker--Planck equation is regularity preserving and that $p_k$ is continuously differentiable in time. For this purpose it is enough to consider one step and to prove that $p_{0}(y_0)\in C_{\mathrm{b}}^{1,\infty}
    ([0,t_1]\times\R^d;\R)$. 
    The conditions of \cite[Theorem 4.3]{cattiaux2002hypoelliptic} are satisfied under \eqref{assumption: coefficient bound and lipschitz} and \eqref{assumption: Hörmander} of Section~\ref{sec:setting}, which tells us that the probability density of $S$ is infinitely smooth and bounded. More precisely, for all $t\in[0,t_1]$ we have $p_{0}(t,y_0)\in C_{\mathrm{b}}^\infty(\R^d;\R)$.
    It remains to show the time regularity. 
    In \cite[Chapter 9]{da2014introduction} the semigroup $(P(t))_{t\ge0}$ of bounded linear operators on $C_{\mathrm{b}}(\R^d;\R)$ is defined as the solution operator to the Kolmogorov backward equation. 
    Likewise, one can define the adjoint semigroup $R = P^*$ on $C_{\mathrm{b}}(\R^d;\R)$, associated with the forward equation \eqref{eq: global Fokker--Planck with update}. 
    From this we have for $t\in[0,t_1]$ that $p_0(t,y_0)=R(t)\phi_0$, where $\phi_0:=p_0(0,y_0)\in C_{\mathrm b}^\infty(\R^d;\R)$. 
    Following \cite[Proposition 9.9]{da2014introduction} one can analogously show that for all $\psi\in C_{\mathrm b}^\infty(\R^d;\R)$ it holds $A^*R(t)\psi=R(t)A^*\psi$. 
    In particular, this holds for $\psi=\phi_0$ and therefore
    \begin{align} \label{proof eq: fokker planck diff}
        \frac{\mathrm d}{\mathrm{d} t}
        p_0(t,y_0)
        =
        \frac{\mathrm d}{\mathrm{d} t}
        R(t)
        \phi_0
        &=
        A^*
        R(t)
        \phi_0
        =
        R(t)
        A^*
        \phi_0,
        \quad
        t\in[0,t_{1}].
    \end{align}
    From the regularity of $\phi_0$ we have that $A^*\phi_0\in C_{\mathrm b}^\infty(\R^d;\R)$ and thus $\frac{\mathrm d}{\mathrm{d} t} p_0(t,y_0)\in C_{\mathrm b}^\infty(\R^d;\R)$ for all $t\in[0,t_1]$.
    This shows that $p_{0}(y_0)\in C_{\mathrm{b}}^{1,\infty}
    ([0,t_1]\times\R^d;\R)$.
    Lastly, the uniqueness of $p_0(y_0)$ follows analogously to \cite[Theorem 9.11]{da2014introduction} and uniqueness of $p_0\in L^\infty(\Y;C([0,t_1]\times\R^d,\R)$ follows since $L$ is bounded and thus so is $\phi_0=L(y)q_0$ over $\Y$.  
\end{proof}
\begin{remark}
The solution to \eqref{eq: global Fokker--Planck with update} gives unnormalized densities. Instead, in \cite{challa2000nonlinear, demissie2016nonlinear} the normalized update step is used, replacing the third row of \eqref{eq: global Fokker--Planck with update} with
\begin{align}
    p_{k}
    (t_k,x,y_{0:k})
    &=
    \frac{
    p_{k-1}
    (t_k,x,y_{0:k-1})
    L(y_k,x)
    }
    {
    \int_{\R^d}
    p_{k-1}
    (t_k,z,y_{0:k-1})
    L(y_k,z)
    \dd z
    }.
\end{align}
The benefit of directly obtaining a normalized density has to be compared to the additional computational cost of evaluating or approximating this integral. In this work, we consider the unnormalized version of the filtering density, but we note that the same methodology holds for the normalized version. This should be seen as a generalization allowing a more flexible framework where normalization is not necessarily required. In practice, as we discuss in Section~\ref{section: numerical discussion} we perform an approximative normalization for numerical stability.
\end{remark}
The scene is now set for applying the deep splitting methodology between each of the measurement updates. However, the derivation in this paper is done in a slightly different way that avoids the explicit splitting equations seen in \cite{bagmark_1,Arnulf_PDE}. This leads to the same approximation scheme in the end, but is beneficial for our error analysis. To prepare for Feynman--Kac representations, used in the the deep splitting scheme and the error analysis, we introduce $X\colon[0,T]\times \Omega \to \mathbb{R}^d$, that for all $t\in [0,T]$, $\mathbb{P}$-a.s., satisfies
\begin{align} 
\label{eq: auxiliary state}
    X_t 
    &= 
    X_0 
    + 
    \int_0^t 
    b
    (X_s)
    \,\mathrm{d} s 
    + 
    \int_0^t 
    \sigma(X_s) 
    \,\mathrm{d} W_s
    ,\quad 
    X_0\sim \widetilde{q}_0.
\end{align}
Here $\widetilde{q}_0$ is a probability density with finite moments. A natural choice is $\widetilde{q}_0=q_0$, but to emphasize that this is not necessary, and since it gives increased flexibility in the resulting optimization problem, we have a general $\widetilde{q}_0$ here. 
From this point, with a slight abuse of notation, we denote by $(t_n)_{n=0}^N$ the time partition between the first two observation times
$t_{0,0}$ and $t_{1,0}$, i.e., $t_n \defeq t_{0,n}$, $n=0,\dots,N$, and the previously defined observation time points $(t_k)_{k=0}^K$ will consistently be denoted by  $(t_{k,0})_{k=0}^K$ to clarify the distinction. We do this to simplify the notation in the formulas and the following proofs. 
We have the following Feynman--Kac representation for the solution $p_{k}$. We remark that here it is sufficient that $\mu$ and $\sigma$ have bounded derivatives.
\begin{proposition}[Feynman--Kac representation formula]
    \label{prop: exact Feynman--Kac}
    The solutions $p_{k}$, $k=0,\dots,K-1$, to \eqref{eq: global Fokker--Planck with update}, in the time points $t_{k,n+1}$, $n=0,\dots,N-1$, and $x\in\R^d$, satisfy
    \begin{align}
    \label{eq: exact Feynman--Kac}
        \begin{split}
        p_{k}
        (t_{k,n+1},x,y_{0:k})
        &=
        \E
        \Big[
        p_{k}
        (t_{k,n},
        X_{t_{N}-t_{n}}^{t_N-t_{n+1},x}
        ,y_{0:k})
        +
        \int_{t_N-t_{n+1}}^{t_N-t_n} 
        F_b
        p_{k}
        (t_{k,N}-s,
        X_s^{t_N-t_{n+1},x}
        ,y_{0:k})
        \dd s
        \Big].
        \end{split}
    \end{align}
\end{proposition}

\begin{proof}
    The proof follows a standard derivation; see, e.g., \cite{klebaner2012introduction}. We start by fixing $y\in \Y$. Furthermore, we simplify the notation by omitting the dependence on $y_{0:k}$, where it is possible to do so without confusion. We start by noting that the solution $p_{k}$ to \eqref{eq: global Fokker--Planck with update}, for $n=0,\dots,N-1$ and $t\in (t_{k,n},t_{k,n+1}]$, satisfies
    \begin{align}
        \frac{\partial}{\partial t}
        p_{k}(t)
        &=
        A_b
        p_{k}(t)
        +
        F_b
        p_{k}(t).
    \end{align}
    Reparametrization of time $t\mapsto t_{k,N}-t$, so that $t_{k,N}-t\in (t_{k,n},t_{k,n+1}]$, yields
    \begin{align} \label{proof eq: backward Kolmogorov eq sect 2}
        \frac{\partial}{\partial t}
        p_{k}(t_{k,N}-t)
        +
        A_b
        p_{k}(t_{k,N}-t)
        &=
        -
        F_b
        p_{k}(t_{k,N}-t)
        ,\quad
        t
        \in
        [t_{k,N}-t_{k,n+1}
        ,
        t_{k,N}-t_{k,n}).
    \end{align}
    Noting that $X$ has the infinitesimal generator $A_b$, and the fact that $p_{k}\in C^{1,2}([t_{k,0},t_{k+1,0}]\times\R^d;\R)$, allows  us to use Itô's formula for $p_{k}(t_{k,N}-t,X_t)$. This
    gives, $\mathbb{P}$-a.s.,  
    \begin{align}
        p_{k}(t_{k,N}-t,X_t)
        &=
        p_{k}(t_{k,n+1},X_{t_{k,N}-t_{k,n+1}})
        +
        \int_{t_{k,N}-t_{k,n+1}}^t
        \big\langle
        \nabla
        p_{k}(t_{k,N}-s,X_s)
        ,
        \sigma(
        X_s)
        \dd W_s
        \big\rangle
        \\
        &\quad+
        \int_{t_{k,N}-t_{k,n+1}}^t
        \Big(
        \frac{\partial}{\partial s}
        p_{k}(t_{k,N}-s,X_s)
        +
        A_b
        p_{k}(t_{k,N}-s,X_s)
        \Big)
        \dd s.
    \end{align}
    Inserting \eqref{proof eq: backward Kolmogorov eq sect 2} into the third term we get
    \begin{align}
        \begin{split} \label{proof eq: ito formula ito integral}
        p_{k}(t_{k,N}-t,X_t)
        &=
        p_{k}(t_{k,n+1},X_{t_{k,N}-t_{k,n+1}})
        +
        \int_{t_{k,N}-t_{k,n+1}}^t
        \big\langle
        \nabla
        p_{k}(t_{k,N}-s,X_s)
        ,
        \sigma(
        X_s)
        \dd W_s
        \big\rangle
        \\
        &\quad-
        \int_{t_{k,N}-t_{k,n+1}}^t
        F_b
        p_{k}(t_{k,N}-s,X_s)
        \dd s.    
        \end{split}
    \end{align}
    We recall that $p_{k}(t_{k,N}-s,y_{0:k})\in C_{\mathrm{b}}^2(\R^d;\R)$ for all $s\in[t_{k,N}-t_{k,n+1},t_{k,N}-t_{k,n})$, and by the assumptions in Section~\ref{sec:setting} we have that $\sigma$ satisfies a linear growth bound and $X$ has finite second moments on $[0,T]$. This guarantees that
    \begin{align}
        \begin{split}
            \int_{t_{k,N}-t_{k,n+1}}^{t_{k,N}-t_{k,n}}
            \E
            \Big[ 
            \big\| 
            \sigma(X_s)^{\top}
            \nabla
            p_{k}
            (t_{k,N}-s,
            X_s)
            \big\|^2 \Big] 
            \, \mathrm{d} s 
            <\infty
        \end{split}
    \end{align}
    and hence the Itô integral in \eqref{proof eq: ito formula ito integral} is a square integrable martingale with respect to $\mathcal{F}_{t_{k,N}-t_{k,n+1}}$.
    Taking the conditional expectation in \eqref{proof eq: ito formula ito integral} we obtain
    \begin{align}
        \begin{split} 
        &
        \E
        \Big[
        p_{k}(t_{k,N}-t,X_t)
        \mid
        \mathcal{F}_{t_{k,N}-t_{k,n+1}}
        \Big]
        \\
        &\hspace{2em}=
        \E
        \Big[
        p_{k}(t_{k,n+1},X_{t_{k,N}-t_{k,n+1}})
        -
        \int_{t_{k,N}-t_{k,n+1}}^t
        F_b
        p_{k}(t_{k,N}-s,X_s)
        \dd s
        \mid
        \mathcal{F}_{t_{k,N}-t_{k,n+1}}
        \Big]
        .    
        \end{split}
    \end{align}
    By reordering the terms and using the fact that $p_{k}(t_{k,n+1},X_{t_{k,N}-t_{k,n+1}})$ is $\mathcal{F}_{t_{k,N}-t_{k,n+1}}$-measurable, we get
    \begin{align}
        \begin{split} \label{proof eq: before applying t limit}
        &\hspace{1em}
        p_{k}(t_{k,n+1},X_{t_{k,N}-t_{k,n+1}})
        =
        \E
        \Big[
        p_{k}(t_{k,N}-t,X_t)
        +
        \int_{t_{k,N}-t_{k,n+1}}^t
        F_b
        p_{k}(t_{k,N}-s,X_s)
        \dd s
        \mid
        \mathcal{F}_{t_{k,N}-t_{k,n+1}}
        \Big]
        .    
        \end{split}
    \end{align}
    We note that the right endpoint limit in $t\in[t_{k,N}-t_{k,n+1},t_{k,N}-t_{k,n})$, for $x\in\R^d$, satisfies
    \begin{align}
        p_{k}(t_{k,N}-t
        ,x
        )
        \to
        p_{k}(t_{k,n}
        ,x
        )
        \quad
        \text{as}\quad
        t\to(t_{k,N}-t_{k,n}),
    \end{align}
    and $\mathbb{P}$-a.s.
    \begin{align}
        X_t
        \to
        X_{t_{k,N}-t_{k,n}}
        \quad
        \text{as}\quad
        t\to(t_{k,N}-t_{k,n}).
    \end{align}
    Combining these we see that the $L^2$-limit of the right hand side of \eqref{proof eq: before applying t limit} satisfies
    \begin{align}
    \label{proof eq: limit RHS sect 2}
        \begin{split}
        &\lim_{t\to(t_{k,N}-t_{k,n})}
        \E
        \bigg[
        \Big|
        p_{k}(t_{k,N}-t,X_t)
        +
        \int_{t_{k,N}-t_{k,n+1}}^t
        F_b
        p_{k}(t_{k,N}-s,X_s)
        \dd s
        \\
        &\hspace{6em}-
        \Big(
        p_{k}(t_{k,n},X_{t_{k,N}-t_{k,n}})
        +
        \int_{t_{k,N}-t_{k,n+1}}^{t_{k,N}-t_{k,n}}
        F_b
        p_{k}(t_{k,N}-s,X_s)
        \dd s
        \Big)
        \Big|^2
        \bigg]
        =
        0
        .
        \end{split}
    \end{align}    
    Inserting the limit from \eqref{proof eq: limit RHS sect 2} in \eqref{proof eq: before applying t limit} we obtain
    \begin{align}
        \begin{split} 
        &p_{k}(t_{k,n+1},X_{t_{k,N}-t_{k,n+1}})
        =
        \E
        \Big[
        p_{k}(t_{k,n},X_{t_{k,N}-t_{k,n}})
        +
        \int_{t_{k,N}-t_{k,n+1}}^{t_{k,N}-t_{k,n}}
        F_b
        p_{k}(t_{k,N}-s,X_s)
        \dd s
        \mid
        \mathcal{F}_{t_{k,N}-t_{k,n+1}}
        \Big]
        .    
        \end{split}
    \end{align}
    Rewriting the conditional expectation with respect to a conditioned process $(X_t^{s,x})_{t\ge s}$ (starting in $x\in\R^d$ at time $s$), and making the $y_{0:k}$-dependence explicit in the notation, we obtain
    \begin{align}
        \begin{split}
        &p_{k}(t_{k,n+1},x,y_{0:k})
        =
        \E
        \Big[
        p_{k}(t_{k,n},X_{t_{k,N}-t_{k,n}}^{t_{k,N}-t_{k,n+1},x},y_{0:k})
        +
        \int_{t_{k,N}-t_{k,n+1}}^{t_{k,N}-t_{k,n}}
        F_b
        p_{k}(t_{k,N}-s,X_s^{t_{k,N}-t_{k,n+1},x},y_{0:k})
        \dd s
        \Big]
        .    
        \end{split}
    \end{align}    
    Here we note that $t_{k,N}-t_{k,n} = t_N - t_n$ for all $k$ and $n$. The proved identity holds for all $y\in \Y$. This completes the proof.
\end{proof}
In the next step, we consider a forward Euler approximation of the integral term in \eqref{eq: exact Feynman--Kac}.
First, we introduce the first order differential operator $G_b$ acting on $\phi\in C^1(\R^d;\R)$ according to 
\begin{align}\label{eq:G}
    (G_b\phi)(x)
    =
    \phi(x)
    +
    \tau
    (
    F_b
    \phi
    )
    (x)
    ,
    \quad
    x\in\R^d.
\end{align}
We define the approximations $\pi_{k,n+1}$, $(k,n)\in\mathcal{I}_{K,N-1}$, at each time step $t_{k,n+1}$, of \eqref{eq: exact Feynman--Kac} by the recursive formula
\begin{align}
\label{eq: forward euler Feynman--Kac}
\begin{split}
    \pi_{k,n+1}
    (x,y_{0:k})
    &=
    \E
    \Big[
    G_b
    \pi_{k,n}
    \big(
    X_{t_{N}-t_{n}}^{t_N-t_{n+1},x},
    y_{0:k}
    \big)
    \Big],
    \\
    &\quad
    x\in\R^d
    ,\ 
    y\in \Y
    ,\ 
    k=0,\dots,K-1,\ 
    n=0,\dots,N-1,
\end{split}
\end{align}
satisfying
\begin{align} \label{eq: forwward euler Feynman--Kac initial}
\begin{split} 
    \pi_{0,0}(x,y_{0})
    &= 
    q_0(x)L(y_0,x),
    \\    
    \pi_{k,0}(x,y_{0:k}) 
    &=
    \pi_{k-1,N}(x,y_{0:k-1})
    L(y_k,x),
    \quad 
    k=1,\dots,K.
\end{split}
\end{align}
Notice that we have $n=0$ in the final update for $k=K$. Thus, $\pi_{k,n}$ is defined for $(k,n)\in\mathcal{I}_{K,N}=\{0,\dots,K-1\}\times \{0,\dots,N\}\cup \{K\}\times\{0\}$.

\begin{proposition}
    \label{proposition: pi C-smooth}
    There exists a unique $\pi_{k,n}\in L^\infty(\Y;C( \R^d;\R))$, $(k,n)\in\mathcal{I}_{K,N}$, satisfying \eqref{eq: forward euler Feynman--Kac} and \eqref{eq: forwward euler Feynman--Kac initial}.
    Furthermore,  $\pi_{k,n}(y)\in C_\mathrm{b}^\infty(\R^d;\R)$ for {$(k,n)\in\mathcal{I}_{K,N}$}, $y\in\Y$.
\end{proposition}
\begin{proof}
    We fix $y\in\Y$ and suppress it from the notation whenever suitable. The lemma is proved by induction. For the base case, with $k=0$ and $n=0$, from the assumption \eqref{assumption: coefficient bound and lipschitz} of Section~\ref{sec:setting} we have $\pi_{0,0}(y_{0}) = q_0 L(y_0)\in C_{\mathrm{b}}^\infty(\R^d;\R)$. 
    For the induction step we take $k\in\{0,1,\dots,K\}$, 
    $n\in \{1,\dots,N\}$,
    and assume that $\pi_{k,n-1}\in C_{\mathrm{b}}^\infty(\R^d;\R)$. 
    We use the fact that $\pi_{k,n}(x) 
    = 
    \E[
    G_b\pi_{k,n-1}
    (
    X_{t_{N}-t_{n-1}}^{t_N-t_{n},x}
    )]$, 
    is the Feynman--Kac solution $u\colon [t_{k,n-1},t_{k,n}]\times \R^d\to\R$, at the left endpoint $t_{k,n-1}$, to the Kolmogorov backward equation 
    \begin{align}\label{eq:KBE}
        \frac{\partial}{\partial t}
        u(t,x) 
        + 
        A_b
        u(t,x) 
        &= 
        0,
        \quad
        t\in[t_{k,n-1},t_{k,n});\quad 
        u(t_{k,n},x)
        =
        (
        G_b
        \pi_{k,n-1}
        )
        (x).
    \end{align}
    As a consequence of the assumption, we have $G_b\pi_{k,n-1}\in C_{\mathrm{b}}^\infty(\R^d;\R)$ since $\mu$, $b$, and $\sigma$ are infinitely smooth. By \cite[Theorem 4.8.6]{Kloeden_Platen} or \cite{Hairer_2015}, $u(t)\in C_{\mathrm{b}}^\infty(\R^d;\R)$ for all $t\in[t_{k,n-1},t_{k,n}]$ and this shows that $\pi_{k,n}\in C_{\mathrm{b}}^\infty(\R^d;\R)$. For the case $n=1$ we stress that, by the definition of $\pi$, we take $\pi_{k,0}=\pi_{k-1,N}L$ and notice that it belongs to $C_{\mathrm{b}}^\infty(\R^d;\R)$. The uniqueness of the solution follows by \cite[Theorem 9.11]{da2014introduction}.
    From the boundedness of $L$ it follows that $\pi_{k,n}\in L^\infty(\Y;C( \R^d;\R))$ and this completes the proof.
\end{proof}

The convergence of this approximation is stated in the following lemma, which we prove in Section~\ref{section: main theorem}. We remark that, while the solution $(p_k)_{k=0}^K$ to \eqref{eq: global Fokker--Planck with update} is independent of the auxiliary drift $b$, the discrete approximations $\pi_{k,n+1}$, $(k,n)\in\mathcal{I}_{K,N-1}$, do depend on $b$.
\begin{lemma}
\label{lemma: pi convergence}
    Let $p_k$, $k=0,\dots,K$, be the solution to \eqref{eq: global Fokker--Planck with update} and $\pi_{k,n}$, {$(k,n)\in\mathcal{I}_{K,N}$}, be the solution to \eqref{eq: forward euler Feynman--Kac} and \eqref{eq: forwward euler Feynman--Kac initial}. If $b$ satisfies \eqref{eq: mu-bar choice}, then
    for $\tau \defeq \frac{T}{KN}\leq 1$, there exists $C\defeq C(T,\mu,\sigma,L,K) > 0 $, such that 
    \begin{align}
    \label{eq: pi convergence}
        \sup_{\substack{
              {(k,n)\in\mathcal{I}_{K,N}}  
            }} 
        \big\|
        p_{k}(t_{k,n})
        -
        \pi_{k,n}
        \big\|_{
        L^\infty(\Y;L^\infty(\R^d;\R))
        }
        \leq
        C
        \tau.
    \end{align}
\end{lemma}
Since the equation for the process $X$ lacks analytical solution in general, we must consider approximations. On the finer time mesh $t_{k,n}$, {$(k,n)\in\mathcal{I}_{K,N-1}$}, we define an Euler--Maruyama approximation $Z$, with $Z_{0,0}\sim \widetilde{q}_0$, by 
\begin{align}
\label{eq: Z EM}
    Z_{k,n+1} 
    = 
    Z_{k,n}
    + 
    b
    (Z_{k,n})(t_{k,n+1}-t_{k,n})
    +
    \sigma(Z_{k,n})
    (W_{t_{k,n+1}}-W_{t_{k,n}}).
\end{align}
For convenience, we let $Z_{n} \defeq Z_{0,n}$, $n=0,\dots,N$. Now, we define the approximations $\overline{\pi}_{k,n+1}$, of $\pi_{k,n+1}$, by replacing $X$ with $Z$ in \eqref{eq: forward euler Feynman--Kac}, i.e., by the recursion
\begin{align}
\label{eq: EM Feynman--Kac}
\begin{split}
    \overline{\pi}_{k,n+1}
    (x,y_{0:k})
    &=
    \E
    \Big[
    G_b
    \overline{\pi}_{k,n}
    (Z_{N-n}^{t_{N}-t_{n+1},x},
    y_{0:k})
    \Big],
    \\
    & \quad
    x\in\R^d,
    \ y\in\Y,
    \ 
    k=0,\dots,K-1, \ 
    n=0,\dots,N-1,
\end{split}
\end{align}
satisfying
\begin{align}
\label{eq: EM Feynman--Kac initial}
\begin{split}    
    \overline{\pi}_{0,0}       
    (x,y_0)
    &= 
    q_0(x)L(y_0,x),\\
    \overline{\pi}_{k,0}(x,y_{0:k}) 
    &=
    \overline{\pi}_{k-1,N}(x,y_{0:k-1})
    L(y_k,x),
    \quad 
    k=1,\dots,K.
\end{split}
\end{align}
The following proposition can be proved identically to Proposition~\ref{proposition: pi C-smooth}, adapting arguments with $Z$ satisfying an SDE with piecewise constant coefficients on each time interval $[t_{k,n},t_{k,n+1})$.
\begin{proposition}\label{proposition: pi-streck smooth} 
    There exists a unique $\overline{\pi}_{k,n}\in L^\infty(\Y;C(\R^d;\R))$, {$(k,n)\in\mathcal{I}_{K,N}$}, satisfying \eqref{eq: EM Feynman--Kac} and \eqref{eq: EM Feynman--Kac initial}. 
    Moreover, $\overline{\pi}_{k,n}(y)\in C_\mathrm{b}^\infty(\R^d;\R)$  for $(k,n)\in\mathcal{I}_{K,N}$, $y\in\Y$.
\end{proposition}
The following lemma states the convergence of this second approximation step, with the proof being addressed in Section~\ref{section: main theorem}.
\begin{lemma}
\label{lemma: pi-streck convergence}
    Let $\pi_{k,n}$, {$(k,n)\in\mathcal{I}_{K,N}$}, be the solution to \eqref{eq: forward euler Feynman--Kac} and \eqref{eq: forwward euler Feynman--Kac initial}, and $\overline{\pi}_{k,n}$, $(k,n)\in\mathcal{I}_{K,N}$, be the solution to \eqref{eq: EM Feynman--Kac} and \eqref{eq: EM Feynman--Kac initial}. 
    If $b$ satisfies \eqref{eq: mu-bar choice}, then for $\tau \defeq \frac{T}{KN}\leq 1$, there exists $C\defeq C(T,\mu,\sigma,L,K) > 0 $, such that 
    \begin{align}
    \label{eq: pi-streck convergence}
        \sup_{\substack{
              {(k,n)\in\mathcal{I}_{K,N}}  
            }} 
        \big\|
        \pi_{k,n}
        -
        \overline{\pi}_{k,n}
        \big\|_{
        L^\infty(\Y;L^\infty(\R^d;\R))
        }
        \leq
        C
        \tau.
    \end{align}
\end{lemma}
Finally, we reach our first main result, namely the convergence of the approximation $\overline{\pi}$ of $p$.
\begin{theorem}
\label{thm: main theorem}
    Let $p_k$, $k=0,\dots,K$, be the solution to \eqref{eq: global Fokker--Planck with update} and $\overline{\pi}_{k,n}$, {$(k,n)\in\mathcal{I}_{K,N}$}, be the solution to \eqref{eq: EM Feynman--Kac} and \eqref{eq: EM Feynman--Kac initial}. If $b$ satisfies \eqref{eq: mu-bar choice}, then for
    $\tau \defeq \frac{T}{KN}\leq 1$, there exists $C\defeq C(T,\mu,\sigma,L,K) > 0 $, such that 
    \begin{align}
    \label{eq: main theorem}
        \sup_{\substack{
              {(k,n)\in\mathcal{I}_{K,N}}  
            }} 
        \big\|
        p_{k}(t_{k,n})
        -
        \overline{\pi}_{k,n}
        \big\|_{
        L^\infty(\Y;L^\infty(\R^d;\R))
        }
        \leq
        C
        \tau.
    \end{align}
\end{theorem}
\begin{proof}
   This follows directly from Lemmas~\ref{lemma: pi convergence} and \ref{lemma: pi-streck convergence} together with the triangle inequality.
\end{proof}

\subsection{Optimization based formulation of the deep splitting scheme}

Here we reformulate the conditional expectation,
formulated as a conditioned process in the proof of Proposition~\ref{prop: exact Feynman--Kac}, as the optimum from a recursive minimization problem. This will allow us to approximate the solution by means of a neural network and stochastic gradient descent. We discuss this approximation further in Section~\ref{sect: numerical evaluation}, in the context of numerical examples. The first optimization problem defined below gives a solution to the recursion \eqref{eq: EM Feynman--Kac} for a fixed $y\in\Y$, and was first introduced in \cite{Arnulf} for stochastic PDE trajectories. In \cite{Crisan_Lobbe} a similar approach was introduced for the Fokker--Planck equation. A scheme based on this formulation requires retraining of the neural network for every observation path and works in an offline setting, where one is only interested in the inference of a single path. We prove this proposition below and an alternative proof can be found in \cite[Proposition 2.7]{Beck}. 
\begin{proposition}
\label{prop: first minimization}
    The functions $\overline{\pi}_{k,n+1}$, $k=0,\dots,K-1$, $n=0,\dots,N-1$, defined in \eqref{eq: EM Feynman--Kac} and \eqref{eq: EM Feynman--Kac initial} satisfy the minimization problems 
    \begin{align} 
    \label{eq: first minimization}
    \begin{split}
        &(\overline{\pi}_{k,n+1}
        (x,y_{0:k}))_{x\in \mathbb{R}^d}
        =
        \mathop{\mathrm{arg\,min}}_{u\in C(\mathbb{R}^d;\mathbb{R})}
        \E
        \Big[ \big| 
        u(Z_{N-(n+1)})  
        -
        G_b
        \overline{\pi}_{k,n}
        (Z_{N-n},y_{0:k})
        \big|^2 \Big]
        ,\quad
        n=0,\dots,N-1
        ,\ y\in\mathbb Y.
    \end{split}
    \end{align}
\end{proposition}
\begin{proof}
    We start by fixing $y\in\Y$, $k=0,\dots,K-1$, $n=0,\dots,N-1$, and introduce the random variables 
    \begin{align*}
        \Phi_{k,n+1}
        &=
        \overline{\pi}_{k,n+1}
        (Z_{N-(n+1)},
        y_{0:k}),
        \\
        \Xi_{k,n}
        &=
        G_b
        \overline{\pi}_{k,n}
        (Z_{N-n},
        y_{0:k}).       
    \end{align*}
    By \eqref{eq: EM Feynman--Kac} we have
    $
        \Phi_{k,n+1}
        =
        \E
        [
        \Xi_{k,n}
        \mid
        \mathfrak{S}(Z_{N-(n+1)})
        ]
    $.
    and by the interpretation of the conditional expectation as an orthogonal $L^2$-projection, see  \cite[Corollary 8.17]{Klenke}, we thus have that $\Phi_{k,n+1}$ is the unique minimizer in
    $
        L^2(\Omega;\mathfrak{S}(Z_{N-(n+1)}))
    $
    of
    $
        \Phi
        \mapsto
        \E
        [
        |
        \Phi
        -
        \Xi_{k,n}
        |^2
        ]
    $.
    Moreover, we have $\Phi_{k,n+1} = u^*(Z_{N-(n+1)})$, where 
    \begin{align} \label{proof eq: argmin L0}
        u^*
        &=
        \mathop{\mathrm{arg\,min}}_{u\in \mathcal{L}^0(\R^d;\R)}
        \E
        \Big[
        \big|
        u(Z_{N-(n+1)})
        -
        \Xi_{k,n}      
        \big|^2
        \Big].
    \end{align}
    We know that $u^* = \overline{\pi}_{k,n+1}$ from the definition of $\Phi_{k,n+1}$ and conclude that $u^*$ is continuous, since $\overline{\pi}_{k,n+1}\in C(\R^d;\R)$ by Proposition~\ref{proposition: pi-streck smooth}. We can thus minimize over continuous functions in \eqref{proof eq: argmin L0}.
\end{proof}
We generalize the minimization problem obtained in Proposition~\ref{prop: first minimization} by making a change of probability measure. 
Instead of fixing $y$ we consider the expectation with respect to a probability measure $\mathbb{P}_{Z}\times \mathbb{P}_{Y}$, where $\mathbb{P}_{Z}$ is the distribution of \eqref{eq: Z EM}, and $\mathbb{P}_{Y}$ can be chosen arbitrarily. However, it is natural to use the measurement model $g$ in the statistical model \eqref{introduction: system of equations}.
The obtained solution becomes a deterministic function over $\R^d\times \mathrm{supp}(\mathbb{P}_Y)$, allowing for instantaneous evaluation of new observation sequences, suitable for an online setting. In Section~\ref{section: numerical experiments} we use neural networks as function approximators and the new optimization problem drops the need of retraining the networks. 
\begin{proposition}
\label{prop: second minimization}
    Let $Y\in L^\infty(\Omega;\mathbb{Y})$ be a random variable distributed with $\mathbb{P}_Y$, 
    $\overline{\pi}_{k,n+1}$, {$(k,n)\in\mathcal{I}_{K,N-1}$} be the solutions to \eqref{eq: EM Feynman--Kac}--\eqref{eq: EM Feynman--Kac initial} and the functions $\widetilde{\pi}_{k,n+1}\colon \R^d\times\mathrm{supp}(\mathbb{P}_Y)\to \R$, {$(k,n)\in\mathcal{I}_{K,N-1}$}, be the solutions to the recursive optimization problems
    \begin{align} 
    \label{eq: 
    second minimization}
    \begin{split}
        (\widetilde{\pi}_{k,n+1}
        (x,y))_{(x,y)\in \mathbb{R}^d\times \mathrm{supp}(\mathbb{P}_Y)}  
        &=
        \mathop{\mathrm{arg\,min}}_{u\in 
        L^\infty(\mathrm{supp}(\mathbb{P}_Y);
        C(\mathbb{R}^d;\mathbb{R}))}
        \mathbb{E}
        \bigg[ \Big| 
        u(Z_{N-(n+1)},Y_{0:{k}})
        -
        G_b
        \widetilde{\pi}_{k,n}
        (Z_{N-n},Y_{0:k})
        \Big|^2 \bigg],
        \\
        &\hspace{4em}
        k=0,\dots,K-1,
        \
        n=0,1,\dots,N-1,
        \\
        \widetilde{\pi}_{0,0}(x,y_{0}) 
        &= 
        q_0(x)L(y_0,x),
        \\
        \widetilde{\pi}_{k,0}(x,y_{0:k}) 
        &=
        \widetilde{\pi}_{k-1,N}(x,y_{0:k-1})
        L(y_k,x),
        \quad 
        k=1,\dots,K.
    \end{split}
    \end{align}
    Then $\widetilde{\pi}$ is well defined and moreover, for all $(k,n)\in\mathcal{I}_{K,N}$, and $y\in\mathrm{supp}(\mathbb{P}_Y)$, we have $\overline{\pi}_{k,n}(y_{0:k}) = \widetilde{\pi}_{k,n}(y_{0:k})$.
\end{proposition}

\begin{proof} 
    We remind the reader that $Y$ now represents a random variable with distribution $\mathbb{P}_{Y}$, and hence it is enough to show that $\overline{\pi}_{k,n}(\cdot,Y_{0:k}) = \widetilde{\pi}_{k,n}(\cdot,Y_{0:k})$ in $L^\infty(\Omega;\R)$.
    We proceed with the proof by using an induction argument.
    The base case $k=0$ and $n=0$ is trivially valid, since $\overline{\pi}_{0,0}(Y_0)=q_0L(Y_0)=\widetilde{\pi}_{0,0}(Y_0)$ by definition, and moreover $x\mapsto q_0(x)L(y,x)$ is continuous for every $y\in\mathbb Y$ and $(x,y)\mapsto q_0(x)L(y,x)$ is bounded.
    For the induction step, we fix $k\in\{0,\dots,K-1\}$, $n\in\{1,\dots,N\}$ and assume that $\widetilde{\pi}_{k,n-1}(Z_{N-(n-1)},Y_{0:k})= \overline{\pi}_{k,n-1}(Z_{N-(n-1)},Y_{0:k})\in L^\infty(\Omega;\R)$. 
    By definition, for a fixed $Y_{0:k}(\omega) = y_{0:k}$, we have
    \begin{align}
        \overline{\pi}_{k,n}
        (Z_{N-n},y_{0:k})
        =
        \E
        \Big[
        G_b
        \overline{\pi}_{k,n-1}
        (Z_{N-(n-1)},y_{0:k})
        \mid
        \mathfrak{S}(Z_{N-n})
        \Big].
    \end{align}
    The minimization problem \eqref{eq: second minimization} is solved by the conditional expectation
    \begin{align}
        &\widetilde{\pi}_{k,n}
        (Z_{N-n},Y_{0:k})
        =        
        \E
        \Big[
        G_b
        \widetilde{\pi}_{k,n-1}
        (Z_{N-(n-1)},Y_{0:k})
        \mid
        \mathfrak{S}
        (Z_{N-n},Y_{0:k})
        \Big].
    \end{align}  
    By the inductive assumption we have
    \begin{align}
        \widetilde{\pi}_{k,n}
        (Z_{N-n},Y_{0:k})
        &=
        \E
        \Big[
        G_b
        \overline{\pi}_{k,n-1}
        (Z_{N-(n-1)},Y_{0:k})
        \mid
        \mathfrak{S}(Z_{N-n},Y_{0:k})
        \Big]\\
        &=
        \E
        \Big[
        G_b
        \overline{\pi}_{k,n-1}
        (Z_{N-(n-1)},y_{0:k})
        \mid
        \mathfrak{S}(Z_{N-n})
        \Big]
        \Bigg|_{y_{0:k}=Y_{0:k}}\\
        &=
        \overline{\pi}_{k,n}(Z_{N-n},Y_{0:k}).
    \end{align}
    Since $x\mapsto\overline{\pi}_{k,n}(x,y)$ is continuous for all $y\in\mathbb Y$, then so is $\widetilde{\pi}_{k,n}$, as we have shown that they coincide for $y\in\mathrm{supp}(\mathbb{P}_Y)\subset \mathbb{Y}$.
\end{proof}
The next result is our second main result. It is a direct consequence of Theorem~\ref{thm: main theorem} and Proposition~\ref{prop: second minimization}.
\begin{corollary}
\label{corollary: pi-tilde convergence}
    Let $p_k$, $k=0,\dots,K$, be the solution to \eqref{eq: global Fokker--Planck with update} and $\widetilde{\pi}_{k,n}$, {$(k,n)\in\mathcal{I}_{K,N}$}, be the solution to \eqref{eq: 
    second minimization}.  
    If $b$ satisfies \eqref{eq: mu-bar choice}, then for $\tau \defeq \frac{T}{KN}\leq 1$, there exists $C\defeq C(T,\mu,\sigma,L,K) > 0 $, such that 
    \begin{align} \label{eq: corllary: pi-tilde convergence}
        \sup_{\substack{
              k=0,\dots,K \\
              n=0,\dots,N  
            }} 
        \big\|
        p_{k}(t_{k,n})
        -
        \widetilde{\pi}_{k,n}
        \big\|_{
        L^\infty(
        \mathrm{supp}(\mathbb{P}_Y);
        L^\infty(\R^d;\R))
        }
        \leq
        C
        \tau.
    \end{align}
\end{corollary}

\subsection{Convergence of deep splitting approximation of the Fokker--Planck equation}
In this section we harvest the fruits of our analysis for the Fokker--Planck equation alone, disconnected from any filtering problem. For this purpose we take $K=1$, denote by a slight abuse of notation, $t_{k,n}$ by $t_n$ and notice that $0=t_0<t_1<\dots<t_N=T$ with uniform time step $\tau=T/N$. The reader can verify that all our proofs are valid for $L\equiv1$, corresponding to no update and only prediction. From Theorem~\ref{thm: main theorem}, Proposition~\ref{proposition: pi-streck smooth}, and Propositions~\ref{prop:properties of p} and \ref{prop: first minimization} we conclude the following theorem.
\begin{theorem}\label{thm:main3}
    Let $p\in C_{\mathrm{b}}^{1,\infty}([0,T]\times \R^d;\R)$ be the solution to the Fokker--Planck equation
    \begin{align*}
        p(t)
        =
        q_0
        +
        \int_0^t
          A^*
          p(s)
        \dd s,
        \quad
        t\in[0,T],
    \end{align*}
    and $\overline{\pi}_{n}\in C_{\mathrm{b}}^{\infty}(\R^d;\R)$, $n=0,\dots,N-1$, defined by the minimization problems 
    \begin{align} 
    \begin{split}
        (\overline{\pi}_{n+1}
        (x))_{x\in \mathbb{R}^d}
        =
        \mathop{\mathrm{arg\,min}}_{u\in C(\mathbb{R}^d;\mathbb{R})}
        \E
        \bigg[ \Big| 
        u(Z_{N-(n+1)})  
        -
        G_b
        \overline{\pi}_{n}
        (Z_{N-n})
        \Big|^2 \bigg]
        ,\quad
        n=0,\dots,N-1
        ,
    \end{split}
    \end{align}
    where $Z_n$, $n=0,\dots,N$, is the Euler--Maruyama approximation
    \begin{align}
    Z_{n+1} 
    = 
    Z_{n}
    + 
    b
    (Z_{n})
    (t_{n+1}-t_{n})
    +
    \sigma(Z_{n})
    (W_{t_{n+1}}-W_{t_{n}}),
    \quad n=0,\dots,N-1,
    \end{align}
    with $Z_0\sim \widetilde{q}_0$. If $b$ satisfies \eqref{eq: mu-bar choice}, then for $\tau \defeq \frac{T}{N}\leq 1$, there exists $C\defeq C(T,\mu,\sigma) > 0 $ such that
    \begin{align}
        \sup_{
              n=0,\dots,N  
            }
        \big\|
        p(t_{n})
        -
        \overline{\pi}_{n}
        \big\|_{
        L^\infty(\R^d;\R)
        }
        \leq
        C
        \tau.
    \end{align}
\end{theorem}   

\section{Convergence proofs}
\label{section: error analysis}
This section is devoted to the error analysis of the approximation scheme obtained in Proposition~\ref{prop: second minimization}. Section~\ref{section: preliminaries for the analysis} contains some preliminaries used for the analysis. In Section~\ref{sec: local convergence} we state and prove two auxiliary lemmas which are used in Section~\ref{section: main theorem} to prove Lemma~\ref{lemma: pi convergence} and Lemma~\ref{lemma: pi-streck convergence}.

\subsection{Preliminaries for the analysis}
\label{section: preliminaries for the analysis}
We remind the reader that throughout this section we assume the setting of Section~\ref{sec:setting} and we let $b$ be defined by \eqref{eq: mu-bar choice}. We recall that this choice of $b$ eliminates the first order term in $F_b$. For convenience in the later proofs we rewrite the operator $F_b$ by defining $f_0\colon\R^d\to\R$ 
\begin{align}
    f_0
    (x) 
    =&\,
    \frac{1}{2}\sum_{i,j=1}^d 
    \frac{\partial^2 a_{ij}(x)}{\partial x_i \partial x_j}
    -
    \sum_{i=1}^d 
    \frac{\partial \mu_i(x)}{\partial x_i}.
\end{align}
With this function, we can write $F_b$, acting on a function $\phi\in C^1(\R^d;\R)$, in the form
\begin{align}
    \label{eq: F1 F2 relation}
    (F_b\phi)(x)
    &
    =
    f_0(x)
    \phi(x)
    ,\quad
    x\in\R^d.
\end{align}
From Assumption~\eqref{assumption: coefficient bound and lipschitz} of Section~\ref{sec:setting}, Proposition~\ref{prop:properties of p}, and Proposition~\ref{proposition: pi-streck smooth}, we define bounding constants $C_\phi = \max_{|\alpha| \leq 4}\|\mathrm{D}^\alpha \phi\|_{L^\infty(\Y;L^\infty(\R^d;\R))}$ for $\phi = f_0,\, L$, and $\overline{\pi}_{k,n}$, for $(k,n)\in\mathcal{I}_{K,N}$.
Since $p_{k}(y)\in C_{\mathrm b}^{1,\infty}([t_{k,0},t_{k+1,0}]\times \R^d; \R))$ for all $y\in\Y$ and $k=0,\dots,K-1$, $p_k$ is Lipschitz in time, i.e., there exists a constant $C_{p}>0$ such that
\begin{align} \label{eq: lipschitz in time}
    \big\|
    p_k(t,y_{0:k})
    -
    p_k(s,y_{0:k})
    \big\|_{L^\infty(\R^d;\R)}
    \leq
    C_p
    |
    t
    -
    s
    |,
    \quad
    s,t\in [t_{k,0},t_{k+1,0}],\hspace{0.5em}
    k=0,\dots,K-1,\hspace{0.5em}
    y\in\Y,
\end{align}
where we note that $
    \|
    \phi
    \|_{L^\infty(\R^d;\R)}
    =
    \|
    \phi
    \|_{C(\R^d;\R)}
$ for $\phi\in C_{\mathrm{b}}(\R^d;\R)$.

\subsection{Local convergence} \label{sec: local convergence}
In preparation for the proofs of Lemma~\ref{lemma: pi convergence} and Lemma~\ref{lemma: pi-streck convergence}, we first state and prove two lemmas. 
All the estimations in the proofs hold uniformly with respect to $y\in\Y$, and when we need to emphasize the spatial dependence we write $p_k(t)(x)$ instead of $p_k(t,x,y)$.
\begin{lemma}
    \label{lemma: recursive lemma 1}
    Let $p_k$, $k=0,\dots,K$, be the solution to \eqref{eq: global Fokker--Planck with update} and $\pi_{k,n}$, {$(k,n)\in\mathcal{I}_{K,N}$}, be the solution to \eqref{eq: forward euler Feynman--Kac} and \eqref{eq: forwward euler Feynman--Kac initial}. There exist $C_3,C_4>0$, such that for all {$(k,n)\in\mathcal{I}_{K,N-1}$}, we have 
    \begin{align}
    \begin{split}
    \label{eq: recursive lemma 1}
        \|
        p_{k}
        (t_{k,n+1})
        -
        \pi_{k,n+1}
        \|_{L^\infty(\Y;L^\infty(\R^d;\R))}&
        \leq 
        C_3
        \tau^{
        2
        }
        +
        (1+C_4\tau)
        \|
        p_{k}
        (t_{k,n})
        -
        \pi_{k,n}
        \|_{L^\infty(\Y;L^\infty(\R^d;\R))}.
    \end{split}
    \end{align}
\end{lemma}
\begin{proof}
Let $y\in\Y$, and {$(k,n)\in\mathcal{I}_{K,N-1}$} be fixed, and let $y$ be hidden in the notation.
We begin by recalling, from Proposition~\ref{prop: exact Feynman--Kac} and \eqref{eq: forward euler Feynman--Kac}, the Feynman--Kac representations 
\begin{align}
    p_{k}(t_{k,n+1})
    (x)
    =
    \E
    \Big[
    p_{k}(t_{k,n})
    (X_{t_{N}-t_{n}}^{t_N-t_{n+1},x})
    +
    \int_{t_N-t_{n+1}}^{t_N-t_n} 
    F_b
    p_{k}
    (t_{k,N}-s)
    (X_s^{t_N-t_{n+1},x})
    \dd s
    \Big]
\end{align}
and
\begin{align}
    \pi_{k,n+1}
    (x)
    =
    \E
    \Big[
    G_b
    \pi_{k,n}
    (X_{t_{N}-t_{n}}^{t_N-t_{n+1},x})
    \Big]
    =
    \E
    \Big[
    \pi_{k,n}
    (X_{t_{N}-t_{n}}^{t_N-t_{n+1},x})
    +
    \tau
    F_b
    \pi_{k,n}
    (X_{t_{N}-t_{n}}^{t_N-t_{n+1},x})
    \Big]
    .
\end{align}
Using these expressions and the triangle inequality we get 
\begin{align}
 \begin{split}
    \|
    p_{k}(t_{k,n+1})
    &
    - 
    {\pi}_{k,n+1}
    \|_{L^\infty(\R^d;\R)}
    \leq
    \sup_{x\in\R^d}
    \Big|
    \E
    \Big[
    p_{k}(t_{k,n})
    (X_{t_{N}-t_{n}}^{t_N-t_{n+1},x})
    -
    \pi_{k,n}
    (X_{t_{N}-t_{n}}^{t_N-t_{n+1},x})
    \Big]
    \Big|
    \\
    &\quad +
    \sup_{x\in\R^d}
    \Big|
    \E
    \Big[
    \int_{t_N-t_{n+1}}^{t_N-t_n} 
    F_b
    p_{k}
    (t_{k,N}-s)
    (X_s^{t_N-t_{n+1},x})
    \dd s
    -
    \tau
    F_b
    \pi_{k,n}
    (X_{t_{N}-t_{n}}^{t_N-t_{n+1},x})
    \Big]
    \Big|
    =
    \text{I}
    +
    \text{II}
    .
    \end{split}
\end{align}
For the first term we use the triangle inequality and take a supremum over $\R^d$ inside the expectation to obtain
\begin{align} \label{proof eq: recursive inequality}
\begin{split}
    \text{I}
    \leq
    \sup_{x\in\R^d}
    \E
    \bigg[
    \Big|
    p_{k}(t_{k,n})
    (X_{t_{N}-t_{n}}^{t_N-t_{n+1},x})
    -
    \pi_{k,n}
    (X_{t_{N}-t_{n}}^{t_N-t_{n+1},x})
    \Big|
    \bigg]
    &\leq
    \sup_{x\in\R^d}
    \E
    \bigg[
    \sup_{z\in\R^d}
    \Big|
    p_{k}(t_{k,n})
    (z)
    -
    \pi_{k,n}
    (z)
    \Big|
    \bigg]
    \\
    =
    \sup_{x\in\R^d}
    \Big|
    p_{k}(t_{k,n})
    (x)
    -
    \pi_{k,n}
    (x)
    \Big|
    &\leq
    \big\|
    p_{k}(t_{k,n})
    -
    \pi_{k,n}
    \big\|_{L^\infty(\Y;L^{\infty}(\R^d;\R))}.
\end{split}
\end{align}
This fits the form of the recursive bound in Lemma~\ref{lemma: recursive lemma 1}. The second term is handled by adding and subtracting $F_bp_{k}(t_{k,n})(X_{t_{N}-t_{n}}^{t_N-t_{n+1},x})\tau$, and applying the triangle inequality to obtain 
\begin{align*}
    \text{II}
    &\leq
    \sup_{x\in\R^d}
    \Big|
    \E
    \Big[
    \int_{t_N-t_{n+1}}^{t_N-t_n} 
    F_b
    p_{k}
    (t_{k,N}-s)
    (X_s^{t_N-t_{n+1},x})
    \dd s
    -
    \tau
    F_b
    p_{k}
    (t_{k,n})
    (X_{t_{N}-t_{n}}^{t_N-t_{n+1},x})
    \Big]
    \Big|
    \\
    &
    \hspace{8em}+
    \sup_{x\in\R^d}
    \Big|
    \E
    \Big[
    \tau
    F_b
    p_{k}
    (t_{k,n})
    (X_{t_{N}-t_{n}}^{t_N-t_{n+1},x})
    -
    \tau
    F_b
    \pi_{k,n}
    (X_{t_{N}-t_{n}}^{t_N-t_{n+1},x})
    \Big]
    \Big|
    =\,
    \text{II}_1
    +
    \text{II}_2
    .
\end{align*}
By substituting $F_b$ as in \eqref{eq: F1 F2 relation}, applying Hölder's inequality, and using the bound $C_{f_0}$, we see that
\begin{align}
    \text{II}_2
    &=
    \tau
    \sup_{x\in\R^d}
    \Big|
    \E
    \Big[
    f_0
    (X_{t_{N}-t_{n}}^{t_N-t_{n+1},x})
    (
    p_{k}
    (t_{k,n})
    -
    \pi_{k,n}
    )
    (X_{t_{N}-t_{n}}^{t_N-t_{n+1},x})
    \Big]
    \Big|
    \\
    &\leq
    C_{f_0}
    \tau
    \sup_{x\in\R^d}
    \E
    \bigg[
    \Big|
    p_{k}(t_{k,n})
    (X_{t_{N}-t_{n}}^{t_N-t_{n+1},x})
    -
    \pi_{k,n}
    (X_{t_{N}-t_{n}}^{t_N-t_{n+1},x})
    \Big|
    \bigg].    
\end{align}
We follow the steps in \eqref{proof eq: recursive inequality} and obtain
\begin{align}
    \text{II}_2
    \leq&\,
    C_{f_0}
    \tau
    \big\|
    p_{k}(t_{k,n})
    -
    \pi_{k,n}
    \big\|_{L^\infty(\Y;L^{\infty}(\R^d;\R))}.
\end{align}
This is of the form required in Lemma~\ref{lemma: recursive lemma 1}. 

It remains to show $\text{II}_1\leq C\tau^{2}$. By the linearity of the integral we have
\begin{align}
    \text{II}_1
    =
    &\,
    \sup_{x\in\R^d}
    \bigg|
    \E
    \Big[
    \int_{t_N-t_{n+1}}^{t_N-t_n}
    \big(
    F_b
    p_{k}
    (t_{k,N}-s)
    (X_s^{t_N-t_{n+1},x})
    -
    F_b
    p_{k}
    (t_{k,n})
    (X_{t_{N}-t_{n}}^{t_N-t_{n+1},x})
    \big)
    \dd s
    \Big]
    \bigg|.
\end{align}
By adding and subtracting $F_bp_{k}(t_{k,n})(X_{s}^{t_N-t_{n+1},x})$ and applying the triangle inequality we obtain
\begin{align}
    \text{II}_1
    &
    \leq
    \sup_{x\in\R^d}
    \bigg|
    \E
    \Big[
    \int_{t_N-t_{n+1}}^{t_N-t_n} 
    \big(
    F_b
    p_{k}
    (t_{k,N}-s)
    (X_s^{t_N-t_{n+1},x})
    -
    F_b
    p_{k}
    (t_{k,n})
    (X_s^{t_N-t_{n+1},x})
    \big)
    \dd s
    \Big]
    \bigg|
    \\
    &\hspace{1em}
    +
    \sup_{x\in\R^d}
    \bigg|
    \E
    \Big[
    \int_{t_N-t_{n+1}}^{t_N-t_n}
    \big(
    F_b
    p_{k}
    (t_{k,n})
    (X_s^{t_N-t_{n+1},x})
    -
    F_b
    p_{k}
    (t_{k,n})
    (X_{t_{N}-t_{n}}^{t_N-t_{n+1},x})
    \big)
    \dd s
    \Big]
    \bigg|
    =
    \text{II}_{1,1}
    +
    \text{II}_{1,2}
    .
\end{align}
We substitute $F_b$ as in \eqref{eq: F1 F2 relation} and use Fubini--Tonelli theorem to observe that
\begin{align}
    \text{II}_{1,1} 
    =&
    \sup_{x\in\R^d}
    \bigg|
    \E
    \Big[
    \int_{t_N-t_{n+1}}^{t_N-t_n} 
    \big(
    f_0
    (X_s^{t_N-t_{n+1},x})
    (
    p_{k}
    (t_{k,N}-s)
    -
    p_{k}
    (t_{k,n})
    )
    (X_s^{t_N-t_{n+1},x})
    \big)
    \dd s
    \Big]
    \bigg|
    \\
    =&
    \sup_{x\in\R^d}
    \bigg|
    \int_{t_N-t_{n+1}}^{t_N-t_n} 
    \E
    \Big[
    f_0
    (X_s^{t_N-t_{n+1},x})
    (
    p_{k}
    (t_{k,N}-s)
    -
    p_{k}
    (t_{k,n})
    )
    (X_s^{t_N-t_{n+1},x})
    \Big]
    \dd s
    \bigg|.
\end{align}
The condition of absolute integrability of the integrand to use Fubini--Tonelli can be verified by using the Cauchy--Schwarz inequality, bound on $f_0$, and Lipschitz property of $p$. 
In the next step we use the bound $C_{f_0}$. Hence, we obtain
\begin{align}\label{proof eq: post semigroup bound}
\begin{split}
    \text{II}_{1,1}
    &\leq
    C_{f_0}
    \sup_{x\in\R^d}
    \bigg(
    \int_{t_N-t_{n+1}}^{t_N-t_n} 
    \E
    \Big[
    \big|
    (
    p_{k}
    (t_{k,N}-s)
    -
    p_{k}
    (t_{k,n})
    )
    (X_s^{t_N-t_{n+1},x})
    \big|^2
    \Big]^\frac{1}{2}
    \dd s
    \bigg).
\end{split}
\end{align}
By a uniform bound and substitution in the time variable we get
\begin{align}
    &
    \int_{t_N-t_{n+1}}^{t_N-t_n} 
    \E
    \Big[
    \big|
    (
    p_{k}
    (t_{k,N}-s)
    -
    p_{k}
    (t_{k,n})
    )
    (X_s^{t_N-t_{n+1},x})
    \big|^2
    \Big]^\frac{1}{2}
    \dd s
    \\
    &\hspace{1em}\leq
    \int_{t_N-t_{n+1}}^{t_N-t_n}
    \big\|
    p_{k}
    (t_{k,N}-s)
    -
    p_{k}
    (t_{k,n})
    \big\|_{L^\infty(\R^d;\R)}
    \dd s
    =
    \int_{t_{k,n}}^{t_{k,n+1}}
    \big\|
    p_{k}
    (s)
    -
    p_{k}
    (t_{k,n})
    \big\|_{L^\infty(\R^d;\R)}
    \dd s \label{proof eq: before mild solution}.
\end{align}
Recalling the Lipschitz continuity in time \eqref{eq: lipschitz in time} we obtain
\begin{align}
    \int_{t_{k,n}}^{t_{k,n+1}}
    \big\|
    p_{k}
    (s)
    -
    p_{k}
    (t_{k,n})
    \big\|_{L^\infty(\R^d;\R)}
    \dd s
    \leq
    C_p
    \int_{t_{k,n}}^{t_{k,n+1}}
    (
    s
    -
    t_{k,n}
    )
    \dd s
    =
    \frac{1}{2}
    C_p 
    \tau^2.
\end{align}
Inserting this in \eqref{proof eq: post semigroup bound} we get
\begin{align}
    \text{II}_{1,1}
    &\leq
    \frac{1}{2}
    C_{f_0}
    C_p 
    \tau^2.
\end{align}
For the term $\text{II}_{1,2}$, we apply Itô's formula to obtain
\begin{align} \label{proof eq: RHS ito on Fp}
\begin{split}
    F_b
    p_{k}
    (t_{k,n})
    (X_s^{t_N-t_{n+1},x})
    -
    F_b
    p_{k}
    (t_{k,n})
    (X_{t_{N}-t_{n}}^{t_N-t_{n+1},x})
    =&\,
    -
    \int_s^{t_N-t_n}
    A_b
    (F_b
    p_k(t_{k,n}))
    (X_r^{t_N-t_{n+1},x})
    \dd r
    \\
    &\hspace{-7em}
    -
    \int_s^{t_{N}-t_{n}}
    \langle
    \nabla
    (F_b
    p_{k}
    (t_{k,n}))
    (X_r^{t_N-t_{n+1},x})
    ,
    \sigma
    (X_r^{t_N-t_{n+1},x})
    \dd W_r
    \rangle.
\end{split}
\end{align}
We note that the second term is a square integrable martingale with respect to $\mathcal{F}_{s}$, since the integrand is uniformly bounded, and hence it vanishes under the conditional expectation. We apply the Fubini--Tonelli theorem to change the order of the conditional expectation and time integral in $\text{II}_{1,2}$, and then insert the right hand side of \eqref{proof eq: RHS ito on Fp} to get
\begin{align}
    \text{II}_{1,2}
    &=
    \sup_{x\in\R^d}
    \bigg|
    \int_{t_N-t_{n+1}}^{t_N-t_n}
    \E
    \Big[
    \big(
    F_b
    p_{k}
    (t_{k,n})
    (X_s^{t_N-t_{n+1},x})
    -
    F_b
    p_{k}
    (t_{k,n})
    (X_{t_{N}-t_{n}}^{t_N-t_{n+1},x})
    \big)
    \Big]
    \dd s
    \bigg|
    \\
    &=
    \sup_{x\in\R^d}
    \bigg|
    \int_{t_N-t_{n+1}}^{t_N-t_n}
    \E
    \Big[
    \int_s^{t_N-t_n}
    A_b
    (F_b
    p_k(t_{k,n}))
    (X_r^{t_N-t_{n+1},x})
    \dd r
    \Big]
    \dd s
    \bigg|
    .
\end{align}
We recall that $p_k(t_{k,n})(y)\in C^\infty_{\mathrm{b}}(\R^d;\R)$ uniformly with respect to $y\in\Y$, hence $A_b(F_bp_k(t_{k,n}))$ is bounded by some constant $C> 0$ depending on $\mu$, $\sigma$ and derivatives of $p$ of order up to two, since $F_b$ is of order zero. We use this bounding constant to obtain
\begin{align}
    \text{II}_{1,2}
    \leq
    \frac{1}{2}
    C
    \tau^2
    .
\end{align}
Thus, we have shown that $\text{II}_{1,2}$, and hence $\text{II}_1$, is bounded as required. Since the obtained bounds are uniform in $y$, we can finish by taking supremum over $y\in\Y$ on the left hand side to obtain \eqref{eq: recursive lemma 1}.
\end{proof}

\begin{lemma}
    \label{lemma: recursive lemma 2}
    Let $\pi_{k,n}$, {$(k,n)\in\mathcal{I}_{K,N}$}, be the solution to \eqref{eq: forward euler Feynman--Kac} and \eqref{eq: forwward euler Feynman--Kac initial}, and $\overline{\pi}_{k,n}$, {$(k,n)\in\mathcal{I}_{K,N}$}, be the solution to \eqref{eq: EM Feynman--Kac} and \eqref{eq: EM Feynman--Kac initial}. There exist $C_5,C_6>0$, such that for all {$(k,n)\in\mathcal{I}_{K,N-1}$}, we have
    \begin{align}
    \begin{split}
    \label{eq: recursive lemma 2}
        \|
        \pi_{k,n+1}
        -
        \overline{\pi}_{k,n+1}
        \|_{L^\infty(\Y;L^\infty(\R^d;\R))}&
        \leq 
        C_5
        \tau^{2}
        +
        (1+C_6\tau)
        \|
        \pi_{k,n}
        -
        \overline{\pi}_{k,n}
        \|_{L^\infty(\Y;L^\infty(\R^d;\R))}.
    \end{split}
    \end{align}
\end{lemma}

\begin{proof}
We start by fixing $y\in\Y$, {$(k,n)\in\mathcal{I}_{K,N-1}$} and show that the convergence rate is independent of $y$. Recalling that the solution $\pi$ to \eqref{eq: forward euler Feynman--Kac} and the solution $\overline{\pi}$ to \eqref{eq: EM Feynman--Kac} satisfy the recursive relations
\begin{align}
    \pi_{k,n+1}
    (x)
    &=
    \E
    \big[
    (G_b
    \pi_{k,n}
    )
    (X_{t_{N}-t_{n}}^{t_N-t_{n+1},x})
    \big],
    \\
    \overline{\pi}_{k,n+1}
    (x)
    &=
    \E
    \big[
    (G_b
    \overline{\pi}_{k,n}
    )
    (Z_{N-n}^{t_N-t_{n+1},x})
    \big],
\end{align}
we obtain
\begin{align}
    &\sup_{x\in\R^d}
    \Big|
    \pi_{k,n+1}(x)
    -
    \overline{\pi}_{k,n+1}(x)
    \Big|
    =
    \sup_{x\in\R^d}
    \Big|
    \E
    \Big[
    (G_b
    \pi_{k,n})
    \big(
    X_{t_N-t_{n}}^{t_N-t_{n+1},x}
    \big)
    -
    (G_b
    \overline{\pi}_{k,n+1})
    \big(
    Z_{N-n}^{t_N-t_{n+1},x}
    \big)
    \Big]
    \Big|.
\end{align}
We add and subtract $(G_b\overline{\pi}_{k,n})(X_{t_N-t_n}^{t_N-t_{n+1},x})$ and apply the triangle inequality to get
\begin{align}
    \sup_{x\in\R^d}
    \Big|
    \pi_{k,n+1}(x)
    -
    \overline{\pi}_{k,n+1}(x)
    \Big|
    \leq&\,
    \sup_{x\in\R^d}
    \Big|
    \E
    \big[
    (G_b
    {\pi}_{k,n})
    (X_{t_N-t_{n}}^{t_N-t_{n+1},x})
    -
    (G_b
    \overline{\pi}_{k,n})
    (X_{t_N-t_{n}}^{t_N-t_{n+1},x})
    \big]
    \Big|
    \\
    &\hspace{0em}+
    \sup_{x\in\R^d}
    \Big|
    \E
    \big[
    (G_b
    \overline{\pi}_{k,n})
    (X_{t_N-t_{n}}^{t_N-t_{n+1},x})
    -
    (G_b
    \overline{\pi}_{k,n})
    (Z_{N-n}^{t_N-t_{n+1},x})
    \big]
    \Big|
    =\;
    \text{I}
    +
    \text{II}
    .
\end{align}
By the definition \eqref{eq:G} of $G_b$ and the triangle inequality we obtain
\begin{align}
    \text{I}
    \leq&\,
    \sup_{x\in\R^d}
    \Big|
    \E
    \big[
    {\pi}_{k,n}
    (X_{t_N-t_{n}}^{t_N-t_{n+1},x})
    -
    \overline{\pi}_{k,n}
    (X_{t_N-t_{n}}^{t_N-t_{n+1},x})
    \big]
    \Big|
    \\
    &+
    \sup_{x\in\R^d}
    \Big|
    \E
    \big[
    \tau
    \big(
    F_b
    {\pi}_{k,n}
    (X_{t_N-t_{n}}^{t_N-t_{n+1},x})
    -
    F_b
    \overline{\pi}_{k,n}
    (X_{t_N-t_{n}}^{t_N-t_{n+1},x})
    \big)
    \big]
    \Big|
    =\,
    \text{I}_1
    +\,
    \text{I}_2
    .
\end{align}
Applying steps analogous to \eqref{proof eq: recursive inequality}, we obtain
\begin{align}
    \label{proof eq: use max norm 1}
    \text{I}_1
    \leq\,
    \sup_{x\in\R^d}
    \E
    \Big[
    \big|
    {\pi}_{k,n}
    (X_{t_N-t_{n}}^{t_N-t_{n+1},x})
    -
    \overline{\pi}_{k,n}
    (X_{t_N-t_{n}}^{t_N-t_{n+1},x})
    \big|
    \Big]
    &\leq
    \|
    {\pi}_{k,n}
    -
    \overline{\pi}_{k,n}
    \|_{L^\infty(\Y;L^{\infty}(\R^d;\R))}.
\end{align}
This is one of the desired terms in the bound in Lemma~\ref{lemma: recursive lemma 2}. For the second term $\text{I}_2$, we want to show
\begin{align}
    \text{I}_{2}
    \leq 
    C
    \tau
    \|
    {\pi}_{k,n}
    -
    \overline{\pi}_{k,n}
    \|_{L^\infty(\Y;L^{\infty}(\R^d;\R))}.
\end{align}
We proceed by substituting $F_b$ as in \eqref{eq: F1 F2 relation}, applying the Cauchy--Schwarz inequality, and using the bound $C_{f_0}$.
This gives
\begin{align}
    \text{I}_{2}
    \leq&\,
    \tau
    \sup_{x\in\R^d}
    \E
    \Big[
    \big|
    f_0(X_{t_N-t_{n}}^{t_N-t_{n+1},x})
    \big|^2
    \Big]^{\frac{1}{2}}
    \E
    \Big[
    \big|
    {\pi}_{k,n}
    (X_{t_N-t_{n}}^{t_N-t_{n+1},x})
    -
    \overline{\pi}_{k,n}
    (X_{t_N-t_{n}}^{t_N-t_{n+1},x})
    \big|^2
    \Big]^{\frac{1}{2}}
    \\
    \leq&\,
    C_{f_0}
    \tau
    \sup_{x\in\R^d}
    \E
    \Big[
    \big|
    {\pi}_{k,n}
    (X_{t_N-t_{n}}^{t_N-t_{n+1},x})
    -
    \overline{\pi}_{k,n}
    (X_{t_N-t_{n}}^{t_N-t_{n+1},x})
    \big|^2
    \Big]^{\frac{1}{2}}.
\end{align}
Finally, using the inequality in \eqref{proof eq: use max norm 1}, we get
\begin{align} \label{proof eq: I_2}
    &
    \text{I}_{2}
    \leq
    C_{f_0}
    \tau
    \|
    {\pi}_{k,n}
    -
    \overline{\pi}_{k,n}
    \|_{L^\infty(\Y;L^{\infty}(\R^d;\R))}
    .
\end{align}
Thus, we have shown 
\begin{align}
    I
    \leq
    \big(
    1
    +
    C_{f_0}
    \tau
    \big)
    \|
    {\pi}_{k,n}
    -
    \overline{\pi}_{k,n}
    \|_{L^\infty(\Y;L^{\infty}(\R^d;\R))}.
\end{align}
The term II is the weak error for the Euler--Maruyama approximation of $X$ with test function $G_b\overline{\pi}_{k,n}$. Compared to the standard setting the error is now only over one time step (from $t_N-t_{n+1}$ to $t_N-t_n$) and it is therefore $O(\tau^2)$ instead of $O(\tau)$. We omit the proof which can be found in, 
e.g., \cite[Theorem 5.3.1]{gobet2016monte}. For a constant $C>0$ that depends on $\mu$, $\sigma$ and $C_{\overline{\pi}}$ we get
\begin{align}
    \text{II}
    \leq
    C
    \tau^{2}.
\end{align}
We conclude the proof by taking supremum over $y\in\Y$ on the left hand side.
\end{proof}
\subsection{Global convergence}
\label{section: main theorem}
This section is devoted to proving the strong convergence in Lemma~\ref{lemma: pi convergence}. The arguments in the following proof can be adapted, using Lemma~\ref{lemma: recursive lemma 2}, to prove Lemma~\ref{lemma: pi-streck convergence}. 

\begin{proof}[Proof of Lemma~\ref{lemma: pi convergence}]
We define $e_{k,n}:=\|p_k(t_{k,n})-\pi_{k,n}\|_{L^\infty(\Y;L^\infty(\R^d;\R))}$, so that Lemma~\ref{lemma: recursive lemma 1} can be summarized, with this notation, as
\begin{align}\label{eq:calc1}
  e_{k,n}
  \leq
  C_3 \tau^{2}
  +
  (1+C_4\tau)e_{k,n-1},
  \quad
  k=0,\dots,K-1,\ 
  n=1,\dots,N.
\end{align}
We next prove a bound at the update time $t_{k,0}$, where we have 
\begin{align}
    p_{k}(t_{k,0},y_{0:k}) 
    -
    \pi_{k,0}(y_{0:k})
    &=
    \big(
    p_{k-1}(t_{k,0},y_{0:k-1})
    -
    \pi_{k-1,N}(y_{0:k-1})
    \big)
    L(y_k).
\end{align}
Using the fact that $t_{k,0} = t_{k-1,N}$ we get
\begin{align}
    &\|
    p_k(t_{k,0},y_{0:k})
    -
    \pi_{k,0}(y_{0:k})
    \|_{L^\infty(\R^d;\R)}
    \leq
    \|
    L(y_k)
    \|_{L^\infty(\R^d;\R)}
    \|
    p_{k-1}(t_{k-1,N},y_{0:k-1})
    -
    \pi_{k-1,N}(y_{0:k-1})
    \|_{L^\infty(\R^d;\R)}.
\end{align}
We take supremum over $\mathbb{Y}$ and use the uniform bound on $L$, to obtain
\begin{align}\label{eq:calc2}
\begin{split}
    e_{k,0}
    &\leq
    C_L
    e_{k-1,N}.
\end{split}
\end{align}
Inequalities \eqref{eq:calc1} and \eqref{eq:calc2} are next used to complete the proof and we begin by fixing {$(k,n)\in\mathcal{I}_{K,N}$}. Repeating \eqref{eq:calc1} over $n$ steps, we get
\begin{align}
    e_{k,n}
    \leq 
    C_3
    \tau^{2}
    \sum_{\ell=0}^{n-1}
    (
    1
    +
    C_4
    \tau
    )^\ell
    +
    (1+C_4\tau)^n
    e_{k,0}
    .
    \label{proof eq: application of recursive lemma}
\end{align}
By recalling $\tau = \frac{T}{KN}$, we note that, for $N\geq 1$,
\begin{align} \label{proof eq: exp bound}
    (
    1
    +
    C_4
    \tau
    )^{N-1}
    \leq
    \exp
    (C_4TK^{-1}).
\end{align}
Applying the formula for geometric sums and \eqref{proof eq: exp bound}, we obtain  
\begin{align*}
    C_3
    \tau^{2}
    \sum_{\ell=0}^{n-1}
    (
    1
    +
    C_4
    \tau
    )^\ell
    &\leq
    C_3
    \tau^{2}
    \sum_{\ell=0}^{N-1}
    (
    1
    +
    C_4
    \tau
    )^\ell
    =
    C_3
    \tau^{2}
    \frac{
    (
    1
    +
    C_4
    \tau
    )^{N-1}
    -
    1
    }{
    (
    1
    +
    C_4
    \tau
    )
    -
    1
    }
    \\
    &\hspace{-4em}=
    C_3 C_4^{-1}
    \big(
    (
    1
    +
    C_4
    \tau
    )^{N-1}
    -
    1
    \big)
    \tau
    \leq
    C_3 C_4^{-1}
    \big(
    \exp
    (C_4TK^{-1})
    -
    1
    \big)
    \tau
    =
    c_1 \tau
    .
\end{align*}
This gives a bound for the first term of \eqref{proof eq: application of recursive lemma}. For the second term we use \eqref{eq:calc2}, \eqref{proof eq: exp bound} to get
\begin{align}
    (1+C_4\tau)^n
    e_{k,0}
    &\leq
    (1+C_4\tau)^N
    e_{k,0}
    \leq
    C_L
    \exp(C_4 TK^{-1})
    e_{k-1,N}
    =
    c_2 e_{k-1,N}. 
    \label{proof eq: proof step in main theorem}
\end{align}
To sum up, we have
\begin{align}
\begin{split}
    e_{k,n}
    \leq
    c_1
    \tau
    +
    c_2
    e_{k-1,N}.
\end{split}
\end{align}
Repeating this procedure $k$ times we obtain
\begin{align}
\begin{split}
    e_{k,n}
    \leq
    c_1
    \tau
    \sum_{\ell=0}^{k-1}
    c_2^{\ell}
    +
    c_2^{k}
    e_{0,0}.
\end{split}
\end{align}
Since $\pi_0 = q_0$ the second term vanishes. We recall that $K$ is finite, while $N$ tends to infinity. Introducing $C = c_1\sum_{\ell=0}^{K}c_2^{\ell}$, we obtain 
\begin{align}
\begin{split} \label{proof eq: p - pi estimate}
    e_{k,n}
    \leq
    c_1
    \tau
    \sum_{\ell=0}^{K}
    c_2^{\ell}
    =
    C
    \tau.
\end{split}
\end{align}
This completes the proof of Lemma~\ref{lemma: pi convergence}.
\end{proof}

\section{Numerical experiments} \label{section: numerical experiments}
In this section we examine the strong convergence order of $\widetilde{\pi}$ numerically. We begin by reminding that the convergence in Corollary~\ref{corollary: pi-tilde convergence} holds for all $y\in \supp(\mathbb{P}_Y)$, and in this section we fix $Y = O$, defined in the statistical model \eqref{introduction: system of equations}, and let $\widetilde{q}_0 = q_0$. First, in Section~\ref{section: numerical approximation}, we detail the final approximation steps needed to perform the numerical experiments. Section~\ref{section: numerical discussion} contains a discussion of numerical design choices. In Section~\ref{sect: numerical evaluation} we describe the approximation of the left hand side in Corollary~\ref{corollary: pi-tilde convergence}. Section~\ref{section: numerical examples} contains the empirical convergence study on two one-dimensional examples. Finally, in Section~\ref{section: highdim example} we demonstrate the method on a nonlinear $10$-dimensional process and compare the performance to classical filter approximations.

\subsection{Energy-based approximation} \label{section: numerical approximation}
In the previous work \cite{bagmark_1}, the Energy-Based Deep Splitting (EBDS) method was employed on the Zakai equation, solving the filtering problem with continuous in time observations. Here we adapt the same approximation procedures. The initial step involves approximating the expectation in \eqref{eq: second minimization} using a Monte Carlo average. This consists of sampling $M$ independent pairs $(Z_{0:N}^m,Y_{0:K}^m)_{m=1}^M$ from \eqref{eq: Z EM} and \eqref{introduction: system of equations} respectively. The updated problem reads
\begin{align} 
    \label{eq: 
    third minimization}
    \begin{split}
        (\widetilde{\pi}_{k,n+1}^M
        (x,y))_{(x,y)\in \mathbb{R}^d\times \mathbb{R}^{d'\times (k+1)}}
        &=
        \mathop{\mathrm{arg\,min}}_{u\in C(\mathbb{R}^d\times \mathbb{R}^{d'\times (k+1)};\mathbb{R})}
        \frac{1}{M}
        \sum_{m=1}^M
        \Big| 
        u(Z_{N-(n+1)}^{m},Y_{0:{k}}^m) 
        - 
        G_b
        \widetilde{\pi}_{k,n}^M
        (Z_{N-n}^m,Y_{0:k}^m)
        \Big|^2,
        \\
        &\hspace{4em}
        k=0,\dots,K-1
        ,\quad
        n=0,\dots,N-1,
        \\
        \widetilde{\pi}_{0,0}^M
        (x,y_{0})
        &= 
        q_0(x)
        L(y_0,x), 
        \\
        \widetilde{\pi}_{k,0}^M
        (x,y_{0:k}) 
        &=
        \widetilde{\pi}_{k-1,N}^M
        (x,y_{0:k-1})
        L(y_k,x),
        \quad
        k=1,\dots,K.
    \end{split}
\end{align}
Approximating expectations with Monte Carlo averages is a classical technique and by the central limit theorem one can, for well behaved integrands, expect a convergence of order $\frac{1}{2}$ in $M$.

The second step is to approximate $C(\mathbb{R}^d\times \mathbb{R}^{d'\times (k+1)};\mathbb{R})$ by a neural network, more precisely, by the space $\{\mathcal{N}_\theta : \theta \in \Theta_k \}\subset C(\R^d\times \R^{d'\times (k+1)};\R)$. Here every $\theta$ in the space of parameters $\Theta_k \subset \R^{(d+d'\times (k+1)) \times J^{\mathfrak{L}} \times 1}$ defines a continuous function $\mathcal{N}_\theta$, a Fully Connected Neural Network (FCNN) with $\mathfrak{L}$ hidden layers and a width of $J$ neurons, mapping $\R^d\times \R^{d'\times (k+1)}$ to $\R$. In this work we use ReLU activation functions between each hidden layer except the final one, for which we instead use an energy-based activation described below. Other constructions with more advanced (or simpler) architectures could be of interest in future work, but here the focus lies on examining the convergence order numerically rather than on optimized architectures. The new optimization problem reads
\begin{align} 
    \label{eq: 
    fourth minimization}
    \begin{split}
        (\widetilde{\pi}_{k,n+1}^{M,\mathfrak{L}}
        (x,y))_{(x,y)\in \mathbb{R}^d\times \mathbb{R}^{d'\times (k+1)}}
        &=
        \mathop{\mathrm{arg\,min}}_{u\in \{\mathcal{N}_\theta: \theta\in\Theta_k\}}
        \frac{1}{M}
        \sum_{m=1}^M
        \Big| 
        u(Z_{N-(n+1)}^{m},Y_{0:{k}}^m)  
        -
        G_b
        \widetilde{\pi}_{k,n}^{M,\mathfrak{L}}
        (Z_{N-n}^m,Y_{0:k}^m)
        \Big|^2
        ,
        \\ 
        &\hspace{4em}
        k=0,\dots,K-1
        ,\quad 
        n=0,\dots,N-1,
        \\
        \widetilde{\pi}_{0,0}^{M,\mathfrak{L}}
        (x,y_{0}) 
        &= 
        q_0(x)
        L(y_0,x)
        ,
        \\
        \widetilde{\pi}_{k,0}^{M,\mathfrak{L}}
        (x,y_{0:k}) 
        &=
        \widetilde{\pi}_{k-1,N}^{M,\mathfrak{L}}
        (x,y_{0:k-1})
        L(y_k,x),
        \quad
        k=1,\dots,K.
    \end{split}
\end{align}
In \cite{siegel2023optimal} it is shown how well FCNNs approximate functions in Sobolev spaces. Specifically, with a fixed width $J$, it can be shown that the approximation converges with order $\mathfrak{L}^{-\gamma}$, where $J$ and $\gamma>0$ depend on the underlying dimension $d$ and the regularity of the Sobolev space. 

The energy-based activation function mentioned above consists of passing a scalar output $f_\theta(x,y_{0:k})$ through the exponential function. That is, we let the normalized conditional density be approximated, for each $\theta\in\Theta_k$ and input pair $(x,y_{0:k})\in \R^d\times \R^{d'\times(k+1)}$, by
\begin{align}
    p
    (x
    \mid
    y_{0:k}
    )
    &\approx
    \frac{
    \mathcal{N}_\theta(x,y_{0:k})
    }{
    Z_\theta(y_{0:k})
    }
\end{align}
where
\begin{align}
    \mathcal{N}_\theta(x,y_{0:k})
    &=
    \mathrm{e}^{-f_\theta(x,y_{0:k})}
    ,
    \quad
    Z_\theta(y_{0:k})
    =
    \int_{\R^d}
    \mathrm{e}^{-f_\theta(z,y_{0:k})}
    \dd z.
\end{align}
There are two main ideas behind the use of energy-based activation for probabilistic modeling. The first is that we do not have to evaluate the normalizing constant $Z_\theta$, if we are only concerned with the unnormalized density.
This speeds up training and still learns the density. The second concerns the scalar energy $f_\theta$. The idea is that it should assign low energy (resulting in a high probability) to pairs $(x,y_{0:k})$, that are more likely to occur and high energy (resulting in low probability) to pairs that are less likely to occur. This fits very well with solving \eqref{eq: fourth minimization}, where it is enough to approximate an unnormalized solution of the Fokker--Planck equation \eqref{eq: global Fokker--Planck with update} and to do the normalization in the inference stage. Trivially, this activation guarantees positive values, which is crucial when approximating probabilities. One can also construct $f_\theta$ in such a way that it guarantees that $Z_\theta$ is finite, e.g., by learning Gaussian tails \cite{bagmark_1}. Finally, we employ the time reparametrization that was done in \cite[Section 3.4]{bagmark_1}, which defines an equivalent optimization problem that is more suitable for computations. 

\subsection{Discussion on numerical approximations}
\label{section: numerical discussion}
In this subsection we collect a few remarks on numerical design choices in the proposed scheme. The first concerns the choice of the auxiliary drift function $b$ in the process $X$. The second concerns how to pass from the unnormalized surrogate $\widetilde{\pi}$ produced by the algorithm to a numerically stable approximation of the normalized filtering density, including practical strategies for approximating the corresponding normalizing constants in moderate and high dimensions.

\subsubsection{Choosing the auxiliary drift function}
In Section~\ref{Section: Method} we introduced the auxiliary process $X$ with drift function $b$. The natural candidate, which was the only choice in the original derivation \cite{Arnulf_PDE}, is to set $b = \mu$. In this case, if $\widetilde{q}_0 = q_0$, then $S$ and $X$ have the same distribution. This choice is natural from a sampling perspective, since the auxiliary process $X$ then explores the same region of the state space as the state $S$. The second choice that we consider is given by~\eqref{eq: mu-bar choice}, which eliminates the first order terms in $F_b$ and therefore allows us to prove strong convergence of order one. A third, not explored, option would be to use a drift coefficient that more efficiently explores the tails of the distribution. For illustration, consider the one-dimensional case where $S$ is a mean-reverting process such that $S_T \sim \mathcal{N}(0,1)$. If we choose an auxiliary drift $b$ with smaller magnitude, thereby weakening the mean reversion, the process $X$ will have heavier tails. Heuristically, this may allow the method to explore the tail behaviour of the desired distribution more efficiently, providing more information about these regions for a fixed number of Monte Carlo samples $M$. A systematic design and analysis of such tail-exploring drifts is beyond the scope of this work, but appears to be an interesting direction for future research.

\subsubsection{Normalization}
In Section~\ref{Section: Method} we introduced the unnormalized filtering density, whose approximation is analyzed in Section~\ref{section: error analysis}. The analysis is valid for both the normalized and unnormalized versions: given an approximation $\widetilde{\pi}_{k,0}^{M,\mathfrak{L}}(x, Y_{0:k}^m)$ of the unnormalized density, an approximation of the corresponding normalized filtering density is obtained by dividing by an approximation of its normalizing constant.

We emphasize that the unnormalized density is also of practical interest. In particular, if one is primarily interested in generating samples from the filtering density, it is often sufficient to know the density up to a multiplicative constant. Since our surrogate $\widetilde{\pi}$ is differentiable in $x$, it can be used directly as a potential in gradient-based Markov chain Monte Carlo methods. For example, Hamiltonian Monte Carlo \cite{duane1987hybrid,brooks2011handbook} can efficiently explore the support of the density by combining Hamiltonian dynamics with an accept-reject step, without ever requiring explicit evaluation of the normalizing constant.

For numerical stability in the training procedure, we nevertheless incorporate an explicit normalization of $\widetilde{\pi}_{k,0}^{M,\mathfrak{L}}(Y_{0:k}^m)$, for $k = 0,\dots,K$, and $m = 1,\dots,M$,
before training each subsequent network. In Section~\ref{section: numerical examples} we evaluate the strong error in two one-dimensional examples, where the normalizing constants can be accurately computed by quadrature. 

In the high-dimensional setting we instead approximate the normalizing constant by importance sampling. Let $q$ be an importance distribution and write the normalizing constant as
\begin{align}
    C^{m,k}
    &=
    \int_{\R^d} 
    \widetilde{\pi}_{k,0}^{M,\mathfrak{L}}(x, Y_{0:k}^m)
    \,\dd x
    =
    \int_{\R^d} 
    \frac{
        \widetilde{\pi}_{k,0}^{M,\mathfrak{L}}(x, Y_{0:k}^m)
    }{
        q(x \mid Y_{0:k}^m)
    }
    q(x \mid Y_{0:k}^m)
    \,\dd x,
    \quad
    m = 1,\dots,M,
    \quad
    k = 0,\dots,K.
\end{align}
To avoid the curse of dimensionality, while still allowing $q$ to depend on $Y_{0:k}^m$, we construct $q$ from an Extended Kalman Filter (EKF) approximation of the filtering distribution, see, e.g., \cite{sarkka2023bayesian}. More precisely, the EKF provides a Gaussian approximation with mean and covariance $(m_k^m, P_k^m)$, and we take
\begin{align}
    q(x 
    \mid 
    Y_{0:k}^m) 
    = 
    \mathcal{N}
    \big(
    x; \,
    m_k^m,\,
    \lambda P_k^m
    \big),
\end{align}
with an inflation factor $\lambda \geq 1$ to ensure that the support of $q$ covers the main mass of $\widetilde{\pi}_{k,0}^{M,\mathfrak{L}}$. The dominant numerical cost in this construction is the inversion of the $d \times d$ covariance matrices, which is at most of order $d^3$ per observation time and therefore avoids the exponential growth associated with grid-based quadrature in high dimensions.

\subsection{Evaluation} \label{sect: numerical evaluation}
The objective of Section~\ref{section: numerical experiments} is to numerically examine the convergence order of the approximation $\widetilde{\pi}$. To do this, we need to approximate the left hand side of \eqref{eq: corllary: pi-tilde convergence}. Subsection~\ref{section: numerical approximation} introduces the approximation $\widetilde{\pi}_{k,n}^{M,\mathfrak{L}}$ of $\widetilde{\pi}_{k,n}$. Now we measure the error between $p_{k}(t_{k,n})$ and $\widetilde{\pi}_{k,n}^{M,\mathfrak{L}}$ in $L^2(\Omega;L^\infty(\R^d;\R))$ for fixed {$(k,n)\in\mathcal{I}_{K,N}$}. In the numerical experiments we thus measure the error in the weaker norm $L^2(\Omega;L^\infty(\mathbb{R}^d;\mathbb{R}))$ rather than in $L^\infty(\supp(\mathbb{P}_Y);L^\infty(\mathbb{R}^d;\mathbb{R}))$. The latter would require a numerical approximation of a supremum over the full support of the observation process; in our examples in Section~\ref{section: numerical examples} we have $\supp(\mathbb{P}_Y) = \mathbb{R}^{10}$, which makes such a computation prohibitively expensive and unstable. By contrast, the $L^2$-norm admits a straightforward Monte Carlo approximation which is much more accurate. By the triangle inequality and by estimating the $L^2$-norm by the $L^\infty$-norm, we have 
\begin{align}
\begin{split} \label{eq: empirical error bound}
    &\big\|
    p_{k}(t_{k,n})
    - 
    \widetilde{\pi}_{k,n}^{M,\mathfrak{L}}
    \big\|_{L^2(\Omega;L^\infty(\R^d;\R))}
    \\
    &\quad
    \leq
    \big\|
    p_{k}(t_{k,n})
    - 
    \widetilde{\pi}_{k,n}
    \big\|_{L^\infty(\supp(\mathbb{P}_Y);L^\infty(\R^d;\R))}
    +
    \big\|
    \widetilde{\pi}_{k,n}
    - 
    \widetilde{\pi}_{k,n}^{M,\mathfrak{L}}
    \big\|_{L^\infty(\supp(\mathbb{P}_Y);L^\infty(\R^d;\R))}.
\end{split}
\end{align}
We see that the first term on the right hand side of \eqref{eq: empirical error bound} is the left hand side in Corollary~\ref{corollary: pi-tilde convergence}. The second term is the error that occurs due to the statistical approximation depending on $M$ and the approximation error by the neural network depending on $\mathfrak{L}$. In Subsection~\ref{section: numerical examples} we set $M$ and $\mathfrak{L}$ sufficiently large such that the first term on the right hand side of \eqref{eq: empirical error bound} is the dominating error. 

We need further approximations to evaluate the left hand side of \eqref{eq: empirical error bound}, which is, for {$(k,n)\in\mathcal{I}_{K,N}$}, defined by
\begin{align}
    \big\|
    p_{k}(t_{k,n})
    -
    \widetilde{\pi}_{k,n}^{M,\mathfrak{L}}
    \big\|_{
    L^2(\Omega;L^\infty(\R^d;\R))
    }
    &=
    \E
    \Big[
    \big\|
    p_{k}(t_{k,n},Y_{0:k})
    -
    \widetilde{\pi}_{k,n}^{M,\mathfrak{L}}
    (Y_{0:k})
    \big\|_{
    L^\infty(\R^d;\R)
    }^2
    \Big]^{\frac{1}{2}}
    \\ 
    \label{eq: true L2Linf error}
    &=
    \E
    \Big[
    \sup_{x\in\R^d}
    \big|
    p_{k}(t_{k,n},x,Y_{0:k})
    -
    \widetilde{\pi}_{k,n}^{M,\mathfrak{L}}
    (x,Y_{0:k})
    \big|^2
    \Big]^{\frac{1}{2}}.
\end{align}
To evaluate \eqref{eq: true L2Linf error}, we first approximate the expectation with $M_{\text{e}}$ Monte Carlo samples from \eqref{introduction: system of equations}. Second, since we are only evaluating this in $d=1$, we approximate the supremum by taking the supremum over a uniform grid $B$ in a bounded domain of $\R$, where the solutions have most of their mass. Combining the approximations, we get 
\begin{align}
\begin{split}
\label{eq: approx L2Linf error}
    &
    \E
    \Big[
    \sup_{x\in\R^d}
    \big|
    p_{k}(t_{k,n},x,Y_{0:k})
    -
    \widetilde{\pi}_{k,n}^{M,\mathfrak{L}}
    (x,Y_{0:k})
    \big|^2
    \Big]^{\frac{1}{2}}
    \\
    &\hspace{2em}\approx
    \Big(
    \frac{1}{M_{\text{e}}}
    \sum_{m_\text{e}=1}^{M_{\text{e}}}
    \sup_{x\in B}
    \big|
    p_{k}(t_{k,n},x,Y_{0:k}^{m_{\text{e}}})
    -
    \widetilde{\pi}_{k,n}^{M,\mathfrak{L}}
    (x,Y_{0:k}^{m_{\text{e}}})
    \big|^2
    \Big)^{\frac{1}{2}}.
\end{split}
\end{align}
The approximation error becomes arbitrarily small by choosing $|B|$ and $M_{\text{e}}$ large enough.

\subsection{Numerical convergence} \label{section: numerical examples}
In this section we empirically examine the convergence order of the EBDS method applied to two one-dimensional examples. The first one, which we refer to as the tanh process, has the tanh function as drift coefficient and satisfies the assumptions in Section~\ref{sec:setting}. The other is a bistable process, characterized by two modes in its corresponding probability distribution, and does not satisfy the assumptions. In both examples, we have $q_0 = \mathcal{N}(0,1)$, $\sigma(x) = 1$, Gaussian measurements $Y_k\sim  \mathcal{N}(h(S_{t_{k,0}}),R)$ with measurement function $h(x) = x$ and variance $R=1$. It is worth mentioning that this corresponds to a rather low signal-to-noise ratio, which some would argue is a harder problem. The tanh process has the drift coefficient $\mu(x) = \tanh(x)$, while the bistable process has a drift coefficient
$\mu(x) = \frac{2}{5}(5x-x^3)$. In both cases, we use $K=20$ with equidistant measurements with a final time of $T=2$. This decision was made based on the observation that the $L^2(L^\infty)$-error reaches a relatively stable value before $t=2$ for both examples. Furthermore, we utilize two different choices of $b$. In the first example we fulfill the assumptions of Corollary~\ref{corollary: pi-tilde convergence} by picking $b$ as in \eqref{eq: mu-bar choice}, and in the second example we use $b=\mu$.

To evaluate \eqref{eq: approx L2Linf error}, we need the reference solution $p_{k}(t_{k,n})$. In both examples we lack analytical solutions and instead employ a bootstrap particle filter with $10^5$ particles, with $128$ auxiliary time steps, to find a good approximation of the true solution. Finally, the constants for the numerical evaluation in the tanh process example and the bistable example are, respectively: $M_{\text{e}} = \{4000,500\}$, $J=128$, $M=10^6$, $\mathfrak{L}=3$ and $|B| = 2\,000$.

In Figure~\ref{fig: over time} we see the error \eqref{eq: approx L2Linf error} with 5 different discretizations $N=1,2,4,8,16$ over time. Each error trajectory corresponds to the average over 10 runs of the method to better illustrate the performance. It is easy to see how the error decreases by an increasing number of time steps $N$ in both examples. On the dashed lines, at time steps $(t_{k,0})_{k=0}^K$, the solutions are updated with new measurements. We note that the overall error level is substantially lower for the tanh process, and in both examples the finer discretizations lead to an earlier plateau of the error.

\begin{figure}[h]
\captionsetup{justification=raggedleft}
\centering
\begin{tikzpicture}

\definecolor{crimson2143940}{RGB}{214,39,40}
\definecolor{darkgray176}{RGB}{176,176,176}
\definecolor{darkorange25512714}{RGB}{255,127,14}
\definecolor{forestgreen4416044}{RGB}{44,160,44}
\definecolor{lightgray204}{RGB}{204,204,204}
\definecolor{mediumpurple148103189}{RGB}{148,103,189}
\definecolor{steelblue31119180}{RGB}{31,119,180}

\begin{axis}[
width=0.45*6.028in,
height=0.45*4.754in,
legend cell align={left},
legend style={
  fill opacity=0.8,
  draw opacity=1,
  text opacity=1,
  at={(0.03,0.97)},
  anchor=north west,
  draw=none
},
legend columns=-1,
tick align=outside,
tick pos=left,
title={Tanh process},
x grid style={darkgray176},
xlabel={$t$},
xmajorgrids,
xmin=-0.1, xmax=2.1,
xtick style={color=black},
y grid style={gray},
ylabel={Average $L^2L^\infty$-error},
ymajorgrids,
ymin=-0.00250299139682389, ymax=0.107452039569865,
ytick style={color=black},
yticklabel style={/pgf/number format/.cd, fixed, precision=2},
legend to name={fig: over time legend}
]
\addplot [line width=0.35mm, mark=pentagon*, mark repeat = 2, mark size=2, darkorange25512714]
table {%
0 0
0.100000001490116 0.0400164840691465
0.200000002980232 0.0359689928160689
0.300000011920929 0.0388758507412132
0.400000005960464 0.0427566666174773
0.5 0.0482360028497986
0.600000023841858 0.0549748328500356
0.699999988079071 0.061147731482683
0.800000011920929 0.0619074798895785
0.900000035762787 0.0670118174735624
1 0.0738228386970021
1.10000002384186 0.0746724068669218
1.20000004768372 0.0787559494488502
1.29999995231628 0.0807919899011077
1.39999997615814 0.0838906332401174
1.5 0.0863218881489343
1.60000002384186 0.0871948616625974
1.70000004768372 0.0879079284193227
1.79999995231628 0.0909445657951747
1.89999997615814 0.093614218612412
2 0.09603615930332
};
\addlegendentry{1}
\addplot [line width=0.35mm, mark=x, mark repeat = 2, mark size=2.7, forestgreen4416044]
table {%
0 0
0.100000001490116 0.012389989524387
0.200000002980232 0.0131246927713778
0.300000011920929 0.0172716421725686
0.400000005960464 0.0228513079475442
0.5 0.0281113872220318
0.600000023841858 0.033396232045775
0.699999988079071 0.03841414410415
0.800000011920929 0.041618881924967
0.900000035762787 0.0470455842476393
1 0.0508308196786659
1.10000002384186 0.0554684167980548
1.20000004768372 0.058950207225192
1.29999995231628 0.0641357349166318
1.39999997615814 0.0649024339664221
1.5 0.0673426890792558
1.60000002384186 0.0680073543297285
1.70000004768372 0.0710095727488622
1.79999995231628 0.0715449304212277
1.89999997615814 0.0745824352430937
2 0.0748441411322684
};
\addlegendentry{2}
\addplot [line width=0.35mm, mark=square*, mark repeat = 2, mark size=1.8, mediumpurple148103189]
table {%
0 0
0.100000001490116 0.0084373614956722
0.200000002980232 0.0099001013013959
0.300000011920929 0.013958133637306
0.400000005960464 0.0167412353576577
0.5 0.0202124163527134
0.600000023841858 0.0245464338598089
0.699999988079071 0.0301955672626818
0.800000011920929 0.0326211267249218
0.900000035762787 0.0369264820300454
1 0.0398750393197464
1.10000002384186 0.0431919903779898
1.20000004768372 0.0454375165116846
1.29999995231628 0.0469812967070382
1.39999997615814 0.0458552012248104
1.5 0.0471272163841891
1.60000002384186 0.0487581962691479
1.70000004768372 0.0507095768186614
1.79999995231628 0.0503298736547105
1.89999997615814 0.0530941468312512
2 0.0542423302635325
};
\addlegendentry{4}
\addplot [line width=0.35mm, mark=*, mark repeat = 2, mark size=1.6, crimson2143940]
table {%
0 0
0.100000001490116 0.0070074723186392
0.200000002980232 0.0089344991506755
0.300000011920929 0.0119838077228646
0.400000005960464 0.0144556457991126
0.5 0.0162488781430538
0.600000023841858 0.0184362340905707
0.699999988079071 0.0215019455544612
0.800000011920929 0.0237902593476014
0.900000035762787 0.0276203096699175
1 0.0295260734275806
1.10000002384186 0.0285865786349653
1.20000004768372 0.0307709654933543
1.29999995231628 0.0295867367418267
1.39999997615814 0.0299102560902509
1.5 0.0315426125186471
1.60000002384186 0.0313904501123344
1.70000004768372 0.0305777073748896
1.79999995231628 0.0308087959221034
1.89999997615814 0.0310331041901254
2 0.0344021651575612
};
\addlegendentry{8}
\addplot [line width=0.35mm, mark=triangle*, mark repeat = 2, mark phase=2, mark size=1.7, steelblue31119180]
table {%
0 0
0.100000001490116 0.006410129338821
0.200000002980232 0.0086519475801655
0.300000011920929 0.0102087248126091
0.400000005960464 0.0118587417921153
0.5 0.0145155443365324
0.600000023841858 0.0161075937276209
0.699999988079071 0.0173634253451647
0.800000011920929 0.0186735492171578
0.900000035762787 0.0199078630206722
1 0.0217550384564015
1.10000002384186 0.0225550320664017
1.20000004768372 0.0219346854601747
1.29999995231628 0.0222779538435627
1.39999997615814 0.0220563169306695
1.5 0.0219450526089224
1.60000002384186 0.0225118713556371
1.70000004768372 0.0228236670657317
1.79999995231628 0.0232418127395399
1.89999997615814 0.0238641600812499
2 0.0217984442126435
};
\addlegendentry{16}
\addplot [line width=0.30mm, black, dashed]
table {%
0 -0.0250299139682389
};
\addlegendentry{Observation time}
\addplot [line width=0.10mm, black, dashed]
table {%
0 -0.0250299139682389
0 0.725636523310804
};
\addplot [line width=0.10mm, black, dashed, forget plot]
table {%
0.1 -0.0250299139682389
0.1 0.725636523310804
};
\addplot [line width=0.10mm, black, dashed, forget plot]
table {%
0.2 -0.0250299139682389
0.2 0.725636523310804
};
\addplot [line width=0.10mm, black, dashed, forget plot]
table {%
0.3 -0.0250299139682389
0.3 0.725636523310804
};
\addplot [line width=0.10mm, black, dashed, forget plot]
table {%
0.4 -0.0250299139682389
0.4 0.725636523310804
};
\addplot [line width=0.10mm, black, dashed, forget plot]
table {%
0.5 -0.0250299139682389
0.5 0.725636523310804
};
\addplot [line width=0.10mm, black, dashed, forget plot]
table {%
0.6 -0.0250299139682389
0.6 0.725636523310804
};
\addplot [line width=0.10mm, black, dashed, forget plot]
table {%
0.7 -0.0250299139682389
0.7 0.725636523310804
};
\addplot [line width=0.10mm, black, dashed, forget plot]
table {%
0.8 -0.0250299139682389
0.8 0.725636523310804
};
\addplot [line width=0.10mm, black, dashed, forget plot]
table {%
0.9 -0.0250299139682389
0.9 0.725636523310804
};
\addplot [line width=0.10mm, black, dashed, forget plot]
table {%
1 -0.0250299139682389
1 0.725636523310804
};
\addplot [line width=0.10mm, black, dashed, forget plot]
table {%
1.1 -0.0250299139682389
1.1 0.725636523310804
};
\addplot [line width=0.10mm, black, dashed, forget plot]
table {%
1.2 -0.0250299139682389
1.2 0.725636523310804
};
\addplot [line width=0.10mm, black, dashed, forget plot]
table {%
1.3 -0.0250299139682389
1.3 0.725636523310804
};
\addplot [line width=0.10mm, black, dashed, forget plot]
table {%
1.4 -0.0250299139682389
1.4 0.725636523310804
};
\addplot [line width=0.10mm, black, dashed, forget plot]
table {%
1.5 -0.0250299139682389
1.5 0.725636523310804
};
\addplot [line width=0.10mm, black, dashed, forget plot]
table {%
1.6 -0.0250299139682389
1.6 0.725636523310804
};
\addplot [line width=0.10mm, black, dashed, forget plot]
table {%
1.7 -0.0250299139682389
1.7 0.725636523310804
};
\addplot [line width=0.10mm, black, dashed, forget plot]
table {%
1.8 -0.0250299139682389
1.8 0.725636523310804
};
\addplot [line width=0.10mm, black, dashed, forget plot]
table {%
1.9 -0.0250299139682389
1.9 0.725636523310804
};
\addplot [line width=0.10mm, black, dashed, forget plot]
table {%
2 -0.0250299139682389
2 0.725636523310804
};
\end{axis}

\end{tikzpicture}
\captionsetup{margin=90pt}
\vspace{0pt}
\centering
\input{include/figures/L2Linf_bimodal}
\captionsetup{margin=80pt}
\vspace{0pt}
\begin{tikzpicture}
\node at (-6.9,-0.4) {\ref*{fig: over time legend}};
\node at (-11.4, -0.4) {$N = $};
\end{tikzpicture}
\vspace{0pt}
\captionsetup{justification = justified}
\caption{The figure illustrates the $L^2(\Omega;L^\infty(\R^d;\R))$-error over time for five different discretizations averaged over 10 instances. To the left we see the error for the tanh process and to the right we see the error for the example with the bistable process. 
}
\label{fig: over time}
\end{figure}

Looking at the final time $T=2$ for the different discretizations, we see the convergence in Figure~\ref{fig: conv plot}. In the figure, we see 10 instances of the performed algorithm, obtaining 10 slightly different approximations due to stochastic gradient descent, illustrated in red. The figure has a logarithmic scale and shows that the average error, illustrated in blue, decreases with order slightly less than $1$ for the tanh process and with about order $\frac{1}{2}$ for the bistable example compared to the reference lines of $O(N^{-1})$ and $O(N^{-\frac{1}{2}})$, where $N=\frac{T}{K\tau}$. 

\begin{figure}[h!]
\captionsetup{justification=raggedleft}
\centering
\begin{tikzpicture}

\definecolor{darkgray176}{RGB}{176,176,176}
\definecolor{lightgray204}{RGB}{204,204,204}

\begin{axis}[
width=0.45*6.028in,
height=0.45*4.754in,
legend cell align={left},
legend style={fill opacity=0.8, draw opacity=1, text opacity=1, draw=none},
legend columns=-1,
log basis x={10},
log basis y={10},
tick align=outside,
tick pos=left,
title={Tanh process},
x grid style={darkgray176},
xlabel={$N$},
xmajorgrids,
xmin=0.840896415253715, xmax=22.0,
xmode=log,
xtick style={color=black},
y grid style={darkgray176},
ylabel={$L^2L^\infty$-error},
ymajorgrids,
ymin=0.01394200050408815, ymax=0.60450208646728,
ymode=log,
ytick style={color=black},
legend to name={fig: conv_legend}
]
\addplot [draw=red, fill=red, mark=x, mark size=2.2, only marks]
table{%
x  y
1 0.0952
1 0.0972
1 0.1064
1 0.0925
1 0.0887
1 0.0894
1 0.0936
1 0.1045
1 0.0973
1 0.1021
2 0.0724
2 0.0737
2 0.0730
2 0.0803
2 0.0748
2 0.0767
2 0.0728
2 0.0766
2 0.0738
2 0.0738
4 0.0568
4 0.0488
4 0.0622
4 0.0525
4 0.0509
4 0.0548
4 0.0461
4 0.0606
4 0.0507
4 0.0596
8 0.0342
8 0.0298
8 0.0316
8 0.0361
8 0.0403
8 0.0350
8 0.0407
8 0.0317
8 0.0331
8 0.0358
16 0.0185
16 0.0236
16 0.0226
16 0.0211
16 0.0230
16 0.0173
16 0.0235
16 0.0249
16 0.0241
16 0.0233
};
\addlegendentry{Instances\,\,}
\addplot [draw=blue, fill=blue, mark=*, mark size=2.5, only marks]
table{%
x  y
1 0.096
2 0.0748
4 0.0542
8 0.0344
16 0.0217
};
\addlegendentry{Average\,\,}
\addplot [semithick, black]
table {%
0.5 1
32 0.015625
};
\addlegendentry{$O(N^{-1})$\,\,}
\addplot [dashed, black]
table {%
1 0.88800295
};
\addlegendentry{$O(N^{-1/2})$}

\end{axis}

\end{tikzpicture}
\captionsetup{margin=90pt}
\vspace{0pt}
\centering
\begin{tikzpicture}

\definecolor{darkgray176}{RGB}{176,176,176}
\definecolor{lightgray204}{RGB}{204,204,204}

\begin{axis}[
width=0.45*6.028in,
height=0.45*4.754in,
legend cell align={left},
legend style={fill opacity=0.8, draw opacity=1, text opacity=1, draw=gray},
legend columns=-1,
log basis x={10},
log basis y={10},
tick align=outside,
tick pos=left,
title={Bistable process},
x grid style={darkgray176},
xlabel={$N$},
xmajorgrids,
xmin=0.840896415253715, xmax=22.0,
xmode=log,
xtick style={color=black},
y grid style={darkgray176},
ymajorgrids,
ymin=0.0394200050408815,
ymax=1.20450208646728,
ymode=log,
ytick style={color=black},
legend to name={fig: conv_legend_2 not used}
]
\addplot [draw=red, fill=red, mark=x, mark size=2.2, only marks]
table{%
x  y
1 0.68015605
2 0.58483016
2 0.5995205
4 0.39920387
8 0.15911894
8 0.20911698
16 0.21394326
16 0.21356715
16 0.18568146
8 0.15944713
4 0.3611002
4 0.38145736
2 0.58652073
1 0.68598765
2 0.56799144
1 0.7067415
1 0.6765037
4 0.27458292
16 0.14544177
8 0.18270564
8 0.23370676
16 0.15029857
4 0.34446543
1 0.7212925
2 0.5944799
4 0.35696125
4 0.33615968
2 0.6081116
1 0.6855605
16 0.15808727
16 0.1740551
8 0.1876449
8 0.17915873
8 0.16904405
16 0.1642633
1 0.669766
2 0.6096589
2 0.6113653
4 0.36368808
4 0.37413302
1 0.6851226
2 0.60184526
8 0.21575952
16 0.18360466
16 0.15501434
8 0.20330502
2 0.58648115
1 0.688107
1 0.70802975
4 0.36158735
};
\addlegendentry{Data}
\addplot [draw=blue, fill=blue, mark=*, mark size=2.5, only marks]
table{%
x  y
1 0.690726725
2 0.595080494
4 0.355333916
8 0.189900767
16 0.174395688
};
\addlegendentry{Average}
\addplot [dashed, black]
table {%
0.5 1.4142
1 1.0
16 0.25
32 0.1768
};
\addlegendentry{$O(N^{-1/2})$}
\end{axis}
\end{tikzpicture}
\captionsetup{margin=80pt}
\vspace{0pt}
\begin{tikzpicture}
\node at (-6.9,-0.4) {\ref*{fig: conv_legend}};
\end{tikzpicture}
\vspace{-2pt}
\captionsetup{justification = justified}
\caption{The figure presents the convergence for the numerical scheme for 10 individual instances of the scheme in red, their average in blue, and in black we see reference lines of order 1 and $\frac{1}{2}$, respectively. To the left we have the errors corresponding to the tanh process and to the right the example with the bistable process. 
}
\label{fig: conv plot}
\end{figure}

This indicates that the assumptions in Section~\ref{sec:setting} are crucial for the numerical convergence order and that order $\frac{1}{2}$ might be possible under weaker assumptions. 

For transparency, we also remark that if we let $M$ and $\mathfrak{L}$ stay constant and keep increasing $N$, the error will stop to decrease after some threshold as the statistical error and neural network approximation error starts to dominate. To examine the convergence we set a sufficiently large sample size $M=10^6$ and layer depth $\mathfrak{L}=3$ for all the discretizations.

\subsection{High-dimensional demonstration}
\label{section: highdim example}
In our final example we consider a high-dimensional version of the tanh process of Section~\ref{section: numerical examples}. More precisely, we let $C$ be a $d\times d$-matrix, with entries $C_{i,j}$, $i,j\in \{1,\dots,d\}$, uniformly distributed in $[-0.5,0.5]$, and let the drift be defined by $\mu_i(x) = \sum_{j=1}^d C_{i,j}\tanh(x_j)$ for $i=1,\dots,d$. Similarly, we let the diffusion coefficient be constant $\sigma(x) = \Sigma$, with a $d\times d$-matrix $\Sigma$, with entries $\Sigma_{i,j}$, $i,j\in \{1,\dots,d\}$, uniformly distributed in $[-0.5,0.5]$. We let $\widetilde{q}_0 = q_0 = \mathcal{N}(0,I)$, and define the observation process by $Y_k\sim \mathcal{N}(h(S_{t_{k,0}}),R)$ with $h(x) = x$ and $R=I$. In this example we let $T=1$, $K=10$ and choose $b=\mu$. Furthermore, in this high-dimensional framework with increasingly small probability values we need to work in log-scale. Hence, we employ the logarithmic version of the method, defined in \cite{baagmark2025high}. To train the method we use $M=10^7$, $J=512$, $\mathfrak{L} = 3$. 

In this high-dimensional nonlinear setting we lack obvious candidates for reference solutions. Instead we consider two metrics, where we do not need a reference solution, but where it suffices to use samples of $S$. For $k=0,\dots,K$, given an approximate mean $\widehat{m}_k$, and an approximate density $\widehat{p}_k$, we consider a Mean Absolute Error (MAE) and a Negative Log-Likelihood (NLL), which are defined, and approximated, by
\begin{align}
    \text{MAE}(\widehat{m}_k)
    &=
    \mathbb{E}
    \Big[
    \big\|
    S_{t_k}
    -
    \widehat{m}_k
    \big\|
    \Big]
    \approx
    \frac{1}{M_{\text{e}}}
    \sum_{m_{\text{e}}=1}^{M_{\text{e}}}
    \big\|
    S_{t_k}^{(m_{\text{e}})}
    -
    m_k^{(m_{\text{e}})}
    \big\|,
    \\
    \text{NLL}
    (\widehat{p}_k)
    &=
    \mathbb{E}
    \Big[
        -
        \log\big(
            \widehat{p}_k
            (S_{t_k}
            \mid 
            Y_{0:k})
        \big)
    \Big]
    \approx
    -
    \frac{1}{M_{\text{e}}}
    \sum_{m_{\text{e}}=1}^{M_{\text{e}}}
    \log\Big(
        \widehat{p}_k
        \big(
        S_{t_k}^{(m_{\text{e}})}
        \mid 
        Y_{0:k}^{(m_{\text{e}})}
        \big)
    \Big).
\end{align} 
By comparing our method to classical filters, such as an Ensemble Kalman Filter (EnKF) \cite{burgers1998analysis} and a bootstrap Particle Filter (PF) \cite{gordon1993novel}, with these metrics, we quantify how well the method is performing. Neither of these metrics are expected to yield $0$ but rather be as small as possible, while the exact filter is the true minimizer of both.

In Figure~\ref{fig: highdim} we report the errors evaluated with $M_{\text{e}} = 10^4$ samples. The figure compares a range of approximative filters, including our proposed method. For the latter we show results for $N=16$ and $N=32$ to highlight the dependence on the auxiliary time discretization within the method. We observe that the method is robust and achieves satisfactory accuracy, with a clear improvement when refining the discretization. Moreover, the Kalman-based methods perform best overall, which is consistent with the fact that the posterior is expected to be close to Gaussian, since the $\tanh$-function is approximately linear on $[-0.5,0.5]$. The particle filter, on the other hand, is asymptotically exact as the number of particles tends to infinity \cite{chopin2004central}, but already in this fully interacting $10$-dimensional example we require on the order of $10^6$ particles to outperform our proposed method in terms of the NLL metric. Finally, we note that our method performs comparatively better in NLL than in MAE, indicating that it captures the overall shape of the filtering density more accurately than its mean.

\begin{figure}[h]
\captionsetup{justification=raggedleft}
\centering
\begin{tikzpicture}

\definecolor{crimson2143940}{RGB}{214,39,40}
\definecolor{darkgray176}{RGB}{176,176,176}
\definecolor{darkorange25512714}{RGB}{255,127,14}
\definecolor{forestgreen4416044}{RGB}{44,160,44}
\definecolor{lightgray204}{RGB}{204,204,204}
\definecolor{mediumpurple148103189}{RGB}{148,103,189}
\definecolor{steelblue31119180}{RGB}{31,119,180}

\begin{axis}[
width=0.45*6.028in,
height=0.45*4.754in,
legend cell align={left},
legend style={
  fill opacity=0.8,
  draw opacity=1,
  text opacity=1,
  at={(0.03,0.97)},
  anchor=north west,
  draw=none
},
legend columns=2,
transpose legend,
tick align=outside,
tick pos=left,
title={MAE},
x grid style={darkgray176},
xlabel={$t$},
xmajorgrids,
xmin=0.05, xmax=1.05,
xtick style={color=black},
y grid style={gray},
ylabel={},
ymajorgrids,
ymin=1.4250299139682389, ymax=2.407452039569865,
ytick style={color=black},
yticklabel style={/pgf/number format/.cd, fixed, precision=2},
legend to name={fig: MAE}
]
\addplot [method LogBSDEF, mark repeat = 3, mark phase = 1]
table {%
0.100000001490116 2.22009524268721
0.200000002980232 1.89607202407703
0.300000011920929 1.73978831180975
0.400000005960464 1.65711436619112
0.5 1.60871392398623
0.600000023841858 1.58018646587679
0.699999988079071 1.55340410177432
0.800000011920929 1.54410383869056
0.899999976158142 1.5395820578139
1 1.53860484595275
};
\addlegendentry{EBDS 16}
\addplot [method BSDEF, mark repeat = 3, mark phase = 3]
table {%
0.100000001490116 2.22024845358115
0.200000002980232 1.88869822504532
0.300000011920929 1.72844479251747
0.400000005960464 1.63929250851348
0.5 1.58990014617767
0.600000023841858 1.5562543443699
0.699999988079071 1.53197217347035
0.800000011920929 1.52237484143607
0.899999976158142 1.5194358831674
1 1.52036889054667
};
\addlegendentry{EBDS 32}
\addplot [method DSF, mark repeat = 3, mark phase = 0]
table {%
0.100000001490116 2.22036347916378
0.200000002980232 1.88689299504362
0.300000011920929 1.72392597749605
0.400000005960464 1.63015729458488
0.5 1.57316338416919
0.600000023841858 1.53491467866466
0.699999988079071 1.50455257461299
0.800000011920929 1.4850923955141
0.899999976158142 1.471620827464
1 1.46001118092082
};
\addlegendentry{EKF}
\addplot [method LogDSF, mark repeat = 3, mark phase = 2]
table {%
0.100000001490116 2.22114418858859
0.200000002980232 1.88798392837371
0.300000011920929 1.725057491106
0.400000005960464 1.63136262150865
0.5 1.5744569283634
0.600000023841858 1.53602945265458
0.699999988079071 1.50582337619072
0.800000011920929 1.48632951238048
0.899999976158142 1.47253244126861
1 1.4611237714039
};
\addlegendentry{EnKF $10^4$ }

\addplot [method EKF, mark repeat = 2, mark phase = 0]
table {%
0.100000001490116 2.24636775165347
0.200000002980232 1.96221228819996
0.300000011920929 1.8327207786953
0.400000005960464 1.74943071753535
0.5 1.69272801504662
0.600000023841858 1.64889552904733
0.699999988079071 1.61318006527484
0.800000011920929 1.58604031531655
0.899999976158142 1.56621746980964
1 1.54900357232022
};
\addlegendentry{PF $10^4$ }

\addplot [method EnKF 1e6, mark repeat = 3, mark phase = 2]
table {%
0.100000001490116 2.22452579910432
0.200000002980232 1.90340237341934
0.300000011920929 1.7525976177436
0.400000005960464 1.66364934276696
0.5 1.60888303464382
0.600000023841858 1.56916425994892
0.699999988079071 1.53711158726084
0.800000011920929 1.51570551958515
0.899999976158142 1.49907279314108
1 1.48398463750005
};
\addlegendentry{PF $10^5$ }

\addplot [method PF 1e6, mark repeat = 3, mark phase = 3]
table {%
0.100000001490116 2.22050644884157
0.200000002980232 1.88987837484734
0.300000011920929 1.73005554184842
0.400000005960464 1.63810530799118
0.5 1.58134772130592
0.600000023841858 1.54294939855834
0.699999988079071 1.5122246316929
0.800000011920929 1.49209670744949
0.899999976158142 1.47806504383758
1 1.46560748018811
};
\addlegendentry{PF $10^6$ }

\end{axis}

\end{tikzpicture}
\captionsetup{margin=90pt}
\vspace{0pt}
\centering
\begin{tikzpicture}

\definecolor{crimson2143940}{RGB}{214,39,40}
\definecolor{darkgray176}{RGB}{176,176,176}
\definecolor{darkorange25512714}{RGB}{255,127,14}
\definecolor{forestgreen4416044}{RGB}{44,160,44}
\definecolor{lightgray204}{RGB}{204,204,204}
\definecolor{mediumpurple148103189}{RGB}{148,103,189}
\definecolor{steelblue31119180}{RGB}{31,119,180}

\begin{axis}[
width=0.45*6.028in,
height=0.45*4.754in,
legend cell align={left},
legend style={
  fill opacity=0.8,
  draw opacity=1,
  text opacity=1,
  at={(0.03,0.97)},
  anchor=north west,
  draw=none
},
legend columns=2,
transpose legend,
tick align=outside,
tick pos=left,
title={NLL},
x grid style={darkgray176},
xlabel={$t$},
xmajorgrids,
xmin=0.05, xmax=1.05,
xtick style={color=black},
y grid style={gray},
ylabel={},
ymajorgrids,
ymin=5.4250299139682389, ymax=12.407452039569865,
ytick style={color=black},
yticklabel style={/pgf/number format/.cd, fixed, precision=2},
legend to name={fig: NLL}
]
\addplot [method LogBSDEF, mark repeat = 3, mark phase = 1]
table {%
0.100000001490116 10.8974697553932
0.200000002980232 9.29553909397604
0.300000011920929 8.39407196715849
0.400000005960464 7.84998307156203
0.5 7.48630706269537
0.600000023841858 7.25414339861079
0.699999988079071 7.06225924755461
0.800000011920929 6.9714374350543
0.899999976158142 6.93489213205462
1 6.92696993554657
};
\addlegendentry{EBDS 16}
\addplot [method BSDEF, mark repeat = 3, mark phase = 3]
table {%
0.100000001490116 10.8980356964035
0.200000002980232 9.25470184441188
0.300000011920929 8.3238150198855
0.400000005960464 7.73825672283844
0.5 7.36297572437842
0.600000023841858 7.10543253433764
0.699999988079071 6.93504840045718
0.800000011920929 6.82808295446425
0.899999976158142 6.79361237473224
1 6.82067024288465
};
\addlegendentry{EBDS 32}
\addplot [method DSF, mark repeat = 3, mark phase = 0]
table {%
0.100000001490116 10.9006594125949
0.200000002980232 9.24481210277308
0.300000011920929 8.29535784553643
0.400000005960464 7.67031653083149
0.5 7.2361178493979
0.600000023841858 6.91886868548753
0.699999988079071 6.65996904708632
0.800000011920929 6.46897037544442
0.899999976158142 6.32200960657704
1 6.19336375519259
};
\addlegendentry{EKF}
\addplot [method LogDSF, mark repeat = 3, mark phase = 2]
table {%
0.100000001490116 11.297701773332
0.200000002980232 9.64189977022871
0.300000011920929 8.6934476306091
0.400000005960464 8.06625257185356
0.5 7.63190741275423
0.600000023841858 7.31642614776765
0.699999988079071 7.05782982093006
0.800000011920929 6.86578245977661
0.899999976158142 6.72342295143473
1 6.59255455007505
};
\addlegendentry{EnKF $10^4$ }


\addplot [method EnKF 1e6, mark repeat = 3, mark phase = 2]
table {%
0.100000001490116 11.5053933780996
0.200000002980232 10.5706174457493
0.300000011920929 9.951604071574
0.400000005960464 9.44710532145284
0.5 9.01838299736905
0.600000023841858 8.65610748799003
0.699999988079071 8.34592138702546
0.800000011920929 8.0822751677815
0.899999976158142 7.84894312566249
1 7.65060610507601
};
\addlegendentry{PF $10^5$ }

\addplot [method PF 1e6, mark repeat = 3, mark phase = 3]
table {%
0.100000001490116 11.2197743276855
0.200000002980232 9.82218908664569
0.300000011920929 8.92328032776339
0.400000005960464 8.30024531378818
0.5 7.82364728342948
0.600000023841858 7.47143683601264
0.699999988079071 7.20450842200811
0.800000011920929 6.97487383511797
0.899999976158142 6.79546568980768
1 6.64494717181028
};
\addlegendentry{PF $10^6$ }

\end{axis}

\end{tikzpicture}
\captionsetup{margin=80pt}
\vspace{0pt}
\begin{tikzpicture}
\node at (-6.9,-0.4) {\ref*{fig: MAE}};
\end{tikzpicture}
\vspace{0pt}
\captionsetup{justification = justified}
\caption{
The figure presents the two errors, mean absolute error to the left, and the negative log-likelihood to the right, for the different approximative filters when applied to the $10$-dimensional Tanh process example. 
}
\label{fig: highdim}
\end{figure}

\section*{Acknowledgements}
The work of {K.B.} and {S.L.} was supported by the Wallenberg AI, Autonomous Systems and Software Program (WASP) funded by the Knut and Alice Wallenberg Foundation. The work of {F.R.} was funded by the Swedish Electromobility Centre (SEC) and partially supported by WASP. The computations were enabled by resources provided by the National Academic Infrastructure for Supercomputing in Sweden (NAISS) at Chalmers e-Commons partially funded by the Swedish Research Council through grant agreement no. 2022-06725.

\bibliographystyle{abbrv}
\bibliography{bib}

\begin{thebibliography}{10}

\bibitem{andersson2023convergence}
K.~Andersson, A.~Andersson, and C.~W. Oosterlee.
\newblock Convergence of a robust deep {FBSDE} method for stochastic control.
\newblock {\em SIAM J. Sci. Comput.}, 45:A226--A255, 2023.

\bibitem{apte2008data}
A.~Apte, C.~K. R.~T. Jones, A.~M. Stuart, and J.~Voss.
\newblock Data assimilation: Mathematical and statistical perspectives.
\newblock {\em Int. J. Numer. Methods Fluids}, 56:1033--1046, 2008.

\bibitem{luk2024learning}
E.~Bach, R.~Baptista, E.~Luk, and A.~Stuart.
\newblock Learning optimal filters using variational inference.
\newblock {\em arXiv:2406.18066}, 2024.

\bibitem{bagmark_1}
K.~B{\aa}gmark, A.~Andersson, and S.~Larsson.
\newblock An energy-based deep splitting method for the nonlinear filtering problem.
\newblock {\em Partial Differ. Equ. Appl.}, 4, 2023.

\bibitem{baagmark2025high}
K.~B{\aa}gmark and F.~Rydin.
\newblock High-dimensional {B}ayesian filtering through deep density approximation.
\newblock {\em arXiv:2511.07261}, 2025.

\bibitem{Arnulf}
C.~Beck, S.~Becker, P.~Cheridito, A.~Jentzen, and A.~Neufeld.
\newblock Deep learning based numerical approximation algorithms for stochastic partial differential equations and high-dimensional nonlinear filtering problems.
\newblock {\em arXiv:2012.01194}, 2020.

\bibitem{Arnulf_PDE}
C.~Beck, S.~Becker, P.~Cheridito, A.~Jentzen, and A.~Neufeld.
\newblock Deep splitting method for parabolic {PDE}s.
\newblock {\em SIAM J. Sci. Comput.}, 43:A3135--A3154, 2021.

\bibitem{Beck}
C.~Beck, S.~Becker, P.~Grohs, N.~Jaafari, and A.~Jentzen.
\newblock Solving the {K}olmogorov {PDE} by means of deep learning.
\newblock {\em J. Sci. Comput.}, 88(3):73, 2021.

\bibitem{blackman1999design}
S.~S. Blackman and R.~Popoli.
\newblock {\em Design and Analysis of Modern Tracking Systems}.
\newblock Artech House Publishers, 1999.

\bibitem{brenner2007mathematical}
S.~C. Brenner and L.~R. Scott.
\newblock {\em The Mathematical Theory of Finite Element Methods}, volume~15 of {\em Texts in Applied Mathematics}.
\newblock Springer, New York, third edition, 2008.

\bibitem{brooks2011handbook}
S.~Brooks, A.~Gelman, G.~Jones, and X.-L. Meng.
\newblock {\em Handbook of Markov Chain Monte Carlo}.
\newblock CRC press, 2011.

\bibitem{burgers1998analysis}
G.~Burgers, P.~J. van Leeuwen, and G.~Evensen.
\newblock Analysis scheme in the ensemble {K}alman filter.
\newblock {\em Mon. Wea. Rev.}, 126(6):1719 -- 1724, 1998.

\bibitem{Cassola}
F.~Cassola and M.~Burlando.
\newblock Wind speed and wind energy forecast through {Kalman} filtering of numerical weather prediction model output.
\newblock {\em Appl. Energy}, 99:154--166, 2012.

\bibitem{cattiaux2002hypoelliptic}
P.~Cattiaux and L.~Mesnager.
\newblock Hypoelliptic non-homogeneous diffusions.
\newblock {\em Probab. Theory Related Fields}, 123:453--483, 2002.

\bibitem{challa2000nonlinear}
S.~Challa and Y.~Bar-Shalom.
\newblock Nonlinear filter design using {F}okker-{P}lanck-{K}olmogorov probability density evolutions.
\newblock {\em IEEE Trans. Aerosp. Electron. Syst.}, 36:309--315, 2000.

\bibitem{chopin2004central}
N.~Chopin.
\newblock Central limit theorem for sequential {M}onte {C}arlo methods and its application to {B}ayesian inference.
\newblock {\em Ann. Statist.}, 32(6):2385--2411, 2004.

\bibitem{corenflos2024particlemala}
A.~Corenflos and A.~Finke.
\newblock {Particle-MALA and Particle-mGRAD: Gradient-based {MCMC} methods for high-dimensional state-space models}.
\newblock {\em arXiv:2401.14868}, 2024.

\bibitem{corenflos2024conditioning}
A.~Corenflos, Z.~Zhao, S.~S{\"a}rkk{\"a}, J.~Sj{\"o}lund, and T.~B. Sch{\"o}n.
\newblock Conditioning diffusion models by explicit forward-backward bridging.
\newblock {\em Proc. 28th Int. Conf. Artif. Intell. Stat. (AISTATS)}, 2024.

\bibitem{Crisan_Lobbe}
D.~Crisan, A.~Lobbe, and S.~Ortiz-Latorre.
\newblock An application of the splitting-up method for the computation of a neural network representation for the solution for the filtering equations.
\newblock {\em Stoch. Partial Differ. Equ.: Anal. Comput.}, 10:1050--1081, 2022.

\bibitem{cui2005comparison}
N.~Cui, L.~Hong, and J.~R. Layne.
\newblock A comparison of nonlinear filtering approaches with an application to ground target tracking.
\newblock {\em Signal Processing}, 85:1469--1492, 2005.

\bibitem{da2014introduction}
G.~Da~Prato.
\newblock {\em Introduction to Stochastic Analysis and {M}alliavin Calculus}, volume~13 of {\em Lecture Notes. Scuola Normale Superiore di Pisa (New Series)}.
\newblock Edizioni della Normale, Pisa, third edition, 2014.

\bibitem{Date}
P.~Date and K.~Ponomareva.
\newblock Linear and non-linear filtering in mathematical finance: a review.
\newblock {\em IMA J.~Manag. Math.}, 22:195--211, 2011.

\bibitem{demissie2016nonlinear}
B.~Demissie, M.~A. Khan, and F.~Govaers.
\newblock {Nonlinear filter design using Fokker-Planck propagator in Kronecker tensor format}.
\newblock In {\em 2016 19th International Conference on Information Fusion (FUSION)}, pages 1--8. IEEE, 2016.

\bibitem{duane1987hybrid}
S.~Duane, A.~D. Kennedy, B.~J. Pendleton, and D.~Roweth.
\newblock Hybrid {M}onte {C}arlo.
\newblock {\em Physics Letters B}, 195(2):216--222, 1987.

\bibitem{Duc_Kuroda}
L.~Duc, T.~Kuroda, K.~Saito, and T.~Fujita.
\newblock Ensemble {Kalman} filter data assimilation and storm surge experiments of tropical cyclone nargis.
\newblock {\em Tellus A}, 67:25941, 2015.

\bibitem{E_2017}
W.~E, J.~Han, and A.~Jentzen.
\newblock Deep learning-based numerical methods for high-dimensional parabolic partial differential equations and backward stochastic differential equations.
\newblock {\em Commun. Math. Stat}, 5:349–380, Nov. 2017.

\bibitem{e2017deep}
W.~E and B.~Yu.
\newblock The deep {R}itz method: A deep learning-based numerical algorithm for solving variational problems.
\newblock {\em Commun. Math. Stat}, 1:1--12, 2018.

\bibitem{finke2023conditional}
A.~Finke and A.~H. Thiery.
\newblock {Conditional sequential Monte Carlo in high dimensions}.
\newblock {\em Ann. Statist.}, 51:437--463, 2023.

\bibitem{frey2022convergence}
R.~Frey and V.~K\"ock.
\newblock Convergence analysis of the deep splitting scheme: the case of partial integro-differential equations and the associated forward backward {SDE}s with jumps.
\newblock {\em SIAM J. Sci. Comput.}, 47(1):A527--A552, 2025.

\bibitem{galanis2006applications}
G.~Galanis, P.~Louka, P.~Katsafados, I.~Pytharoulis, and G.~Kallos.
\newblock {Applications of Kalman filters based on non-linear functions to numerical weather predictions}.
\newblock {\em Ann. Geophys}, 24:1--10, 2006.

\bibitem{gobet2016monte}
E.~Gobet.
\newblock {\em Monte-{C}arlo Methods and Stochastic Processes}.
\newblock CRC Press, Boca Raton, FL, 2016.

\bibitem{goodman1997mathematics}
I.~R. Goodman, R.~P.~S. Mahler, and H.~T. Nguyen.
\newblock {\em Mathematics of Data Fusion}.
\newblock Kluwer Academic Publishers Group, Dordrecht, 1997.

\bibitem{gordon1993novel}
N.~J. Gordon, D.~J. Salmond, and A.~F.~M. Smith.
\newblock Novel approach to nonlinear/non-{G}aussian {B}ayesian state estimation.
\newblock {\em IEEE Proceedings F (Radar and Signal Processing)}, 140(2):107--113, 1993.

\bibitem{Hairer_2015}
M.~Hairer, M.~Hutzenthaler, and A.~Jentzen.
\newblock Loss of regularity for {K}olmogorov equations.
\newblock {\em Ann. Probab.}, 43:468--527, 2015.

\bibitem{han2024deep}
J.~Han, W.~Hu, J.~Long, and Y.~Zhao.
\newblock Deep {P}icard iteration for high-dimensional nonlinear {PDE}s.
\newblock {\em SIAM J. Sci. Comput.}, 48(1):C1--C24, 2026.

\bibitem{hu2024score}
Z.~Hu, Z.~Zhang, G.~E. Karniadakis, and K.~Kawaguchi.
\newblock Score-based physics-informed neural networks for high-dimensional {F}okker-{P}lanck equations.
\newblock {\em SIAM J. Sci. Comput.}, 47(3):C680--C705, 2025.

\bibitem{Kalman_Bucy}
R.~E. Kalman and R.~S. Bucy.
\newblock New results in linear filtering and prediction theory.
\newblock {\em J. Basic Eng.}, 83:95--108, 1961.

\bibitem{klebaner2012introduction}
F.~C. Klebaner.
\newblock {\em Introduction to Stochastic Calculus with Applications}.
\newblock Imperial College Press, London, third edition, 2012.

\bibitem{Klenke}
A.~Klenke.
\newblock {\em Probability Theory}.
\newblock Universitext. Springer, London, second edition, 2014.

\bibitem{Kloeden_Platen}
P.~E. Kloeden and E.~Platen.
\newblock {\em Numerical Solution of Stochastic Differential Equations}, volume~23 of {\em Applications of Mathematics (New York)}.
\newblock Springer-Verlag, Berlin, 1992.

\bibitem{Kushner}
H.~J. Kushner.
\newblock On the differential equations satisfied by conditional probablitity densities of {M}arkov processes, with applications.
\newblock {\em J. Soc. Industrial Appl. Math., Series A: Control}, 2:106--119, 1964.

\bibitem{lobbe2022deep}
A.~Lobbe.
\newblock Deep learning for the {B}enes filter.
\newblock {\em Stochastic Transport in Upper Ocean Dynamics}, 10:195--210, 2023.

\bibitem{Lu_2021}
L.~Lu, P.~Jin, G.~Pang, Z.~Zhang, and G.~E. Karniadakis.
\newblock {Learning nonlinear operators via DeepONet based on the universal approximation theorem of operators}.
\newblock {\em Nat. Mach. Intell.}, 3:218–229, 2021.

\bibitem{naesseth2019high}
C.~A. Naesseth, F.~Lindsten, and T.~B. Sch{\"o}n.
\newblock {High-dimensional filtering using nested sequential Monte Carlo}.
\newblock {\em IEEE Trans. Signal Process.}, 67:4177--4188, 2019.

\bibitem{Quinn}
J.~Quinn.
\newblock A high-dimensional particle filter algorithm.
\newblock {\em arXiv:1901.10543}, 2019.

\bibitem{raissi2019physics}
M.~Raissi, P.~Perdikaris, and G.~E. Karniadakis.
\newblock Physics-informed neural networks: {A} deep learning framework for solving forward and inverse problems involving nonlinear partial differential equations.
\newblock {\em J.~Comput. Phys.}, 378:686--707, 2019.

\bibitem{Rebeschini_2015}
P.~Rebeschini and R.~van Handel.
\newblock Can local particle filters beat the curse of dimensionality?
\newblock {\em Ann. Appl. Probab.}, 25:2809--2866, 2015.

\bibitem{Rutzler}
W.~Rutzler.
\newblock Nonlinear and adaptive parameter estimation methods for tubular reactors.
\newblock {\em Ind. {E}ng. {C}hem. {R}es.}, 26:325--333, 1987.

\bibitem{sarkka2023bayesian}
S.~S\"arkk\"a and L.~Svensson.
\newblock {\em Bayesian Filtering and Smoothing}, volume~17 of {\em Institute of Mathematical Statistics Textbooks}.
\newblock Cambridge University Press, Cambridge, second edition, 2023.

\bibitem{schauer2017guided}
M.~Schauer, F.~van~der Meulen, and H.~van Zanten.
\newblock Guided proposals for simulating multi-dimensional diffusion bridges.
\newblock {\em Bernoulli}, 23(4A):2917--2950, 2017.

\bibitem{siegel2023optimal}
J.~W. Siegel.
\newblock {Optimal approximation rates for deep ReLU neural networks on Sobolev and Besov spaces}.
\newblock {\em J. Mach. Learn. Res.}, 24:1--52, 2023.

\bibitem{snyder2011particle}
C.~Snyder.
\newblock Particle filters, the “optimal” proposal and high-dimensional systems.
\newblock In {\em Proceedings of the ECMWF Seminar on Data Assimilation for atmosphere and ocean}, pages 1--10, 2011.

\bibitem{snyder2015performance}
C.~Snyder, T.~Bengtsson, and M.~Morzfeld.
\newblock Performance bounds for particle filters using the optimal proposal.
\newblock {\em Mon. Weather Rev.}, 143:4750--4761, 2015.

\bibitem{Zakai}
M.~Zakai.
\newblock On the optimal filtering of diffusion processes.
\newblock {\em Zeitschrift f{\"u}r Wahrscheinlichkeitstheorie und verwandte Gebiete}, 11:230--243, 1969.

\bibitem{zeng2013state}
Y.~Zeng and S.~Wu, editors.
\newblock {\em State-Space Models}.
\newblock Statistics and Econometrics for Finance. Springer, New York, 2013.

\bibitem{zhao2024conditional}
Z.~Zhao, Z.~Luo, J.~Sj\"{o}lund, and T.~B. Sch\"{o}n.
\newblock Conditional sampling within generative diffusion models.
\newblock {\em Philos. Trans. R. Soc. A}, 383(2299):20240329, 2025.

\end{thebibliography}

\end{document}